\documentclass[a4paper]{article}

\usepackage{amsmath}
\usepackage{amssymb}
\usepackage{amsthm}
\usepackage{enumitem}
\usepackage{tikz}
\usepackage{todonotes}

\usepackage[mathlines]{lineno}
\usepackage[colorlinks=true, allcolors=blue]{hyperref}

\usetikzlibrary{decorations}
\usetikzlibrary{decorations.pathmorphing}
\usetikzlibrary{decorations.pathreplacing}

\theoremstyle{plain}
\newtheorem{theorem}{Theorem}[section]
\newtheorem{lemma}[theorem]{Lemma}
\newtheorem{proposition}[theorem]{Proposition}
\newtheorem{corollary}[theorem]{Corollary}
\newtheorem{conjecture}[theorem]{Conjecture}
\newtheorem{question}[theorem]{Question}

\theoremstyle{definition}
\newtheorem{definition}[theorem]{Definition}

\newcommand{\core}{\ensuremath{\operatorname{core}}}

\newcommand{\allones}{\mathbf{1}}

\newcommand{\ladseg}[1]{$#1$-rung ladder segment}

\newcommand{\mate}[2]{\ensuremath{\Gamma_{#1}(#2)}}

\title{Structure of Cubic Lehman Matrices}

\author{Dillon Mayhew\\
School of Mathematics and Statistics\\
Victoria University of Wellington\\
\texttt{dillon.mayhew@vuw.ac.nz}\\[2ex]
Irene Pivotto, Gordon Royle\\
School of Mathematics and Statistics\\
University of Western Australia\\
\texttt{piv8irene@gmail.com}, \texttt{gordon.royle@uwa.edu.au}}

\begin{document}

\maketitle

\begin{abstract}
A pair $(A,B)$ of square $(0,1)$-matrices is called a \emph{Lehman pair} if $AB^T=J+kI$ for some integer $k\in\{-1,1,2,3,\ldots\}$. In this case $A$ and $B$ are called \emph{Lehman matrices}. This terminology arises because Lehman showed that the rows with the fewest ones in any non-degenerate minimally nonideal (mni) matrix $M$ form a square Lehman submatrix of $M$. Lehman matrices with $k=-1$ are essentially equivalent to \emph{partitionable graphs} (also known as $(\alpha,\omega)$-graphs), so have been heavily studied as part of attempts to directly classify minimal imperfect graphs. In this paper, we view a Lehman matrix as the bipartite adjacency matrix of a regular bipartite graph, focusing in particular on the case where the graph is cubic. From this perspective, we identify two constructions that generate cubic Lehman graphs from smaller Lehman graphs. The most prolific of these constructions involves repeatedly replacing suitable pairs of edges with a particular $6$-vertex subgraph that we call a $3$-rung ladder segment. Two decades ago, L\"{u}tolf \& Margot initiated a computational study of mni matrices and constructed a catalogue containing (among other things) a listing of all cubic Lehman matrices with $k =1$ of order up to $17 \times 17$.  We verify their catalogue (which has just one omission), and extend the computational results to $20 \times 20$ matrices. Of the $908$ cubic Lehman matrices (with $k=1$) of order up to $20 \times 20$, only two do not arise from our $3$-rung ladder construction. However these exceptions can be derived from our second construction, and so our two constructions cover all known cubic Lehman matrices with $k=1$.
\end{abstract}

\section{Introduction}

This paper is concerned with certain square $(0,1)$-matrices that we call \emph{Lehman matrices}\footnote{As detailed below, this terminology differs slightly from that of some previous authors.}, which are defined in the following way.

\begin{definition}
\label{lehmandef}
A pair $(A,B)$ of square $(0,1)$-matrices of the same order is called a \emph{Lehman pair} if $AB^T = J + kI$ for some integer $k\in\{-1,1,2,3,\ldots\}$, where $J$ is the all-ones matrix. An individual matrix is called a \emph{Lehman matrix} if it is in a Lehman pair.
\end{definition}

% For $k=-1$ these matrices are equivalent to the \emph{partitionable graphs} or $(\alpha,\omega)$-graphs that the search for minimally imperfect graphs predating the eventual proof of the Strong Perfect Graph Theorem \cite{MR2233847}. For $k > 0$, these matrices arise naturally in combinatorial optimization, and were also studied by Bridges \& Ryser \cite{BR69} as a generalization of the matrix equation defining the incidence matrix of a combinatorial design. 

We say that a $(0,1)$-matrix is $r$-regular if each of its rows and columns sum to $r$. If $(A,B)$ is a Lehman pair, then there are integers $r$ and $s$ so that $A$ is $A$ is $r$-regular, $B$ is $s$-regular, and $k = rs-n$. This was proved by Bridges and Ryser \cite{BR69} for $k>0$. The analogous statement for $k=-1$ follows from work of Padberg \cite{MR0340053} on minimally imperfect graphs, but was given explicitly in matrix form by Chv\'atal, Graham, Perold and Whitesides \cite{MR535235}. In either case, we say that $A$ has \emph{type} $(n,r,s)$ or just that $A$ is an $(n,r,s)$-Lehman matrix (and so $B$ is an $(n,s,r)$-Lehman matrix). 

If $k = -1$, then we say that the Lehman pair is \emph{negative} and that $A$ and $B$ are negative Lehman matrices, and analogously a Lehman pair and its matrices are \emph{positive} if $k > 0$. However we alert the reader to the fact that graphs essentially equivalent to negative Lehman matrices have previously been extensively studied as \emph{partitionable graphs} $(\alpha,\omega)$ partitionable graphs or simply \emph{$(\alpha,\omega)$-graphs}. (We will elaborate on these connections below, but details may be found in Boros, Gurvich and Hougardy \cite{MR1936943} and the references therein.)  As we wish to emphasize the parallels between positive and negative Lehman matrices and the constructions that create positive Lehman matrices from negative ones (and vice versa), we will still use the term ``negative Lehman matrix'' in the context of this paper.

A small Lehman pair is shown in Figure~\ref{fig:fano}; in this case $A$ is the point-line incidence matrix of the Fano plane and $B = A$ is the same matrix. In this example, $r = s = 3$ and $k = 2$.

\begin{figure}[htb]
\begin{center}
\small
\setlength\arraycolsep{3pt}
\[
\left[\begin{array}{rrrrrrr}
1 & 1 & 0 & 1 & 0 & 0 & 0 \\
0 & 1 & 1 & 0 & 1 & 0 & 0 \\
0 & 0 & 1 & 1 & 0 & 1 & 0 \\
0 & 0 & 0 & 1 & 1 & 0 & 1 \\
1 & 0 & 0 & 0 & 1 & 1 & 0 \\
0 & 1 & 0 & 0 & 0 & 1 & 1 \\
1 & 0 & 1 & 0 & 0 & 0 & 1
\end{array}\right]
\left[\begin{array}{rrrrrrr}
1 & 0 & 0 & 0 & 1 & 0 & 1 \\
1 & 1 & 0 & 0 & 0 & 1 & 0 \\
0 & 1 & 1 & 0 & 0 & 0 & 1 \\
1 & 0 & 1 & 1 & 0 & 0 & 0 \\
0 & 1 & 0 & 1 & 1 & 0 & 0 \\
0 & 0 & 1 & 0 & 1 & 1 & 0 \\
0 & 0 & 0 & 1 & 0 & 1 & 1
\end{array}\right]
=
\left[\begin{array}{rrrrrrr}
3 & 1 & 1 & 1 & 1 & 1 & 1 \\
1 & 3 & 1 & 1 & 1 & 1 & 1 \\
1 & 1 & 3 & 1 & 1 & 1 & 1 \\
1 & 1 & 1 & 3 & 1 & 1 & 1 \\
1 & 1 & 1 & 1 & 3 & 1 & 1 \\
1 & 1 & 1 & 1 & 1 & 3 & 1 \\
1 & 1 & 1 & 1 & 1 & 1 & 3
\end{array}\right]
\]
\end{center}
\caption{The Fano plane gives a $(7,3,3)$-Lehman matrix}
\label{fig:fano}
\end{figure}

Both positive and negative Lehman matrices originally arose in the context of two areas of research generally viewed as part of combinatorial optimization, namely the search for \emph{minimal imperfect graphs} and for \emph{minimal nonideal clutters}. In the next two subsections we give a brief overview of the background and prior literature in each case.

\subsection{Minimal Imperfect Graphs}

A graph is \emph{perfect} if it, and all of its induced subgraphs, have chromatic number equal to the size of the maximum clique. In the 1960s, Berge famously conjectured that a graph is perfect if and only if it does not have an odd cycle or the complement of an odd cycle as an induced subgraph. An equivalent formulation of the Strong Perfect Graph Conjecture is that the only \emph{minimal imperfect graphs} (i.e., minimal with respect to induced subgraphs) are the odd cycles and their complements. As a result, there is a substantial body of work dedicated to elucidating more and more properties of minimal imperfect graphs, with the intention of proving Berge's conjecture by showing that these conditions were sufficiently onerous that they could be met only by odd cycles and their complements.

% The connection to negative Lehman matrices arises through the concept of a \emph{partitionable} grap$(\alpha,\omega)$-graph. 
%While this approach did not succeed on its own, the extensive  it happens, the Strong Perfect Graph Theorem was eventually proved by Chudnovsky, Robertson, Seymour and Thomas \cite{MR2233847} 

Lov\'asz \cite{MR0302480} initiated this line of enquiry by showing that if $G$ is a minimal imperfect graph with maximum clique of size $\omega(G)$ and maximum coclique of size $\alpha(G)$, then $G$ has the following properties: $|V(G)| = \alpha(G) \omega(G) + 1$ and for every $v \in V(G)$, the set $V(G) \backslash \{v\}$ can be partitioned into $\alpha(G)$ cliques of size $\omega(G)$ and $\omega(G)$ cocliques of size $\alpha(G)$.

Padberg \cite{MR0340053} derived a number of even stronger properties, including the fact that a minimal imperfect graph $G$ must have exactly $|V(G)|$ cliques of size $\omega(G)$ and $|V(G)|$ cocliques of size $\alpha(G)$, and that each $\omega(G)$-clique is disjoint from a unique $\alpha(G)$-coclique. Although Padberg's properties appear very restrictive, Bland, Huang and Trotter \cite{MR534949} showed that
they satisfied by a much larger class of graphs, namely the class of \emph{partitionable graphs}, which are defined as follows: An $(r,s)$-partitionable graph $G$ is a graph for which there are integers $r$, $s$ so that $|V(G)| = rs+1$, and for any vertex $v$, the set $V(G) \backslash \{v\}$ has a partition into $r$ cliques of size $s$ and $s$ cocliques of size $r$. They showed that in this case $r = \alpha(G)$ and $s=\omega(G)$, and so such graphs are also known as $(\alpha,\omega)$-partitionable graphs, or just $(\alpha, \omega)$-graphs. 

Chv\'atal, Graham, Perold and Whitesides \cite{MR535235} expressed these results in matrix terms, observing that associated with any $(\alpha,\omega)$-graph, there are two $0/1$-matrices $A$, $B$ satisfying $AB^T = J-I$. The rows of $A$ and $B$ are the characteristic vectors of the cliques of $G$, and the cocliques of $G$, respectively. Conversely, given a negative Lehman matrix, the graph whose edge set is the union of cliques on the support of each row is an $(\alpha,\omega)$-graph. As it may be possible to add one or more edges to an $(\alpha,\omega)$-graph without changing the set of maximum cliques or cocliques, several different $(\alpha,\omega)$ graphs may yield the same pair of matrices. To avoid this minor irritation, an $(\alpha, \omega)$-graph is called \emph{normalized} if every edge lies in a maximum clique, and then normalized $(\alpha,\omega)$-graphs are essentially the same as negative Lehman matrices.

Although the existence of $(\alpha,\omega)$-graphs sharing so many stringent properties with minimal imperfect graphs complicated the situation, it was still the case that characterising $(\alpha,\omega)$-graphs would resolve the Strong Perfect Graph Conjecture. Consequently a steady stream of theoretical and computational results appeared, variously determining additional properties of $(\alpha,\omega)$-graphs (e.g. Seb\H{o} \cite{MR1402142}), finding all small $(\alpha, \omega)$-graphs (e.g. Lam, Swiercz, Thiel and Regener \cite{MR593715},  Whitesides \cite{MR676544} and Boros, Gurvich and Hougardy \cite{MR1936943}), and finding new constructions for $(\alpha, \omega)$-graphs \cite{MR1936943}. 

 Amongst all these results, it is the paper of Boros, Gurvich and Hougardy \cite{MR1936943} that is the most directly relevant to our work. In addition to giving cryptomorphic definitions of partitionable graphs in graph-theoretic, matrix-theoretic and geometric terms, they devised a recursive technique for constructing larger $(\alpha,\omega)$-graphs from smaller ones. Although phrased in entirely different language, our $3$-rung ladder extension applied to negative Lehman graphs is just a rediscovery of their recursive technique. The concepts and terminology can more or less be directly translated between the two settings, at least for cubic graphs --- for example, what we call a ``$3$-rung ladder'' appears in their paper as a ``gem''. Although ultimately equivalent, some concepts and structures are easier to manipulate in $(\alpha,\omega)$-graphs than in our bipartite graphs, and vice versa. With more structural constraints available, it is easier to compute $(\alpha,\omega)$-graphs than negative cubic Lehman graphs. In particular, Boros, Gurvich and Hougardy \cite{MR1936943} constructed all $(\alpha, 3)$-graphs on up to $22$ vertices, which is one case further than the extent of our computations. Fortunately, in all the cases of overlap our numbers agreed.

Eventually, the Strong Perfect Graph Theorem was proved by Chudnovsky, Robertson, Seymour and Thomas \cite{MR2233847}, and although their work built on some of the known properties of minimal imperfect graphs, it primarily used techniques from structural graph theory. In particular, it did not use or create a characterisation of $(\alpha,\omega)$-graphs. Despite this, there is still interest in the class of $(\alpha,\omega)$-graphs, and it is not impossible that a characterisation of them could yet lead to an alternative proof of the Strong Perfect Graph Theorem.

\subsection{Minimal Nonideal Clutters}

The connection to combinatorial optimization arises from attempts to classify \emph{minimally nonideal clutters}. Here, a \emph{clutter} (also known as a \emph{Sperner family}) is a pair $\mathcal{C} =(V,E)$ where $V$ is a finite set and $E \subseteq 2^V$ is a set of subsets of $V$ such that no element of $E$ contains another. The elements of $V$ are usually called the \emph{vertices} of the clutter, and those of $E$ the \emph{hyperedges} (or just \emph{edges}) of the clutter. A clutter can be represented by a $(0,1)$-matrix, with rows indexed by $E$, columns indexed by $V$, and where each row is the incidence vector of the corresponding hyperedge.  Conversely, any $(0,1)$-matrix with the property that there is no row whose support contains the support of another row is a \emph{clutter matrix} (i.e., the matrix of some clutter). We will often blur the distinction between a clutter and its matrix. 

If $\mathcal{C}$ is a clutter with an $m \times n$ clutter matrix $A$, then $\mathcal{C}$ (and also $A$) is called \emph{ideal} if the polyhedron
$
Q(A) = \{x \in \mathbb{R}^n : A x \succeq \mathbf{1} \text{ and } x \succeq \mathbf{0}\} 
$
has \emph{integral} vertices. Here $\mathbf{0}$ and $\mathbf{1}$ represent the all-$0$ and all-$1$ vectors respectively and $\succeq$ indicates that the inequality holds for each coordinate. If a clutter matrix $A$ is ideal, then any \emph{integer} program with coefficient matrix $A$ has the same solutions as its \emph{linear} program relaxation (where the integer requirement is dropped). As integer programs are computationally hard to solve and linear programs computationally feasible, this is a desirable situation, and hence one that we wish to better understand. 

There are notions of \emph{deletion} and \emph{contraction}, and hence \emph{minors}, for clutters that are reminiscent of the same notions for graphs or matroids. If $\mathcal{C}$ is a clutter and $v$ is a vertex of $\mathcal{C}$, then $\mathcal{C} \backslash v$ ($\mathcal{C}$ \emph{delete} $v$) is the clutter with vertex set $V \backslash \{v\}$ whose hyperedges are the hyperedges of $\mathcal{C}$ that do not contain $v$. The clutter $\mathcal{C} / v$ ($\mathcal{C}$ \emph{contract} $v$) is the clutter with vertex set $V \backslash \{v\}$ whose hyperedges are the minimal sets (under inclusion) of the form $H \backslash \{v\}$ where $H$ is a hyperedge of $\mathcal{C}$. In matrix terms, if $A$ is a clutter matrix and $c$ a column, then $A\backslash c$ is obtained by deleting any row that contains a $1$ in column $c$, and then deleting the entire column. The contraction $A/c$ is produced by first deleting column $c$, and then deleting any rows whose support is no longer minimal under set inclusion. Any clutter (or clutter matrix) obtained by a possibly-empty sequence of deletions and contractions is a \emph{minor} of the original clutter (or clutter matrix).

Any minor of an ideal clutter is itself ideal, which raises the possibility of an \emph{excluded-minor} characterisation of ideal clutters. Thus we define a clutter, or a clutter matrix, to be \emph{minimally nonideal} (mni) if it is not ideal, but every proper minor is ideal. The \emph{weight} of a $0/1$-vector is the number of ones in the vector. Lehman \cite{Leh89} proved the seminal result that if $A$ is an mni clutter matrix, then either $A$ belongs to a particular sporadic family (the degenerate projective planes) or the rows of $A$ of minimum weight form a (positive) Lehman matrix as defined in Definition~\ref{lehmandef}. Therefore we may assume that the first $n$ rows of any mni clutter matrix of order $m \times n$ form a positive Lehman matrix.
This raises the possibility of a two-stage approach to understanding mni matrices, namely first characterise Lehman matrices and then understand how additional rows can be added to a Lehman matrix to form a larger mni matrix. Unfortunately, this latter step appears to be extremely difficult because the property of being mni does not behave nicely under addition of rows. In particular, it is possible that adding a row to an mni matrix may result in one that is not mni, and conversely. Cornu\'ejols and Guenin \cite{MR1922337} give a readable and comprehensive treatment of ideal clutters that provides useful additional background and a wider context to this work than we have given here. 

More than 20 years ago, L\"utolf \& Margot \cite{LM98} conducted a computational search based on these observations in order to provide a collection of small mni matrices. They observed that \emph{``we lack a good understanding of the structure of mni matrices''}, and hoped to provide a significant number of examples of mni matrices in the hope that further study would shed light on their structure. For particular values of $r$, they implemented an \emph{orderly algorithm} \cite{MR0491273} to produce a complete list (up to permutations of rows and columns) of $r$-regular $(0,1)$-matrices and then extracted the Lehman matrices from this list. They identified the Lehman matrices that are already mni (without adding any rows) and used a heuristic search to produce non-square mni matrices by adding additional rows to each Lehman matrix. Their results mostly cover the cases where $r=3$, the matrices have order at most $17 \times 17$, and $k = 1$. The constraints on size and valency are consequences of the very rapid increase in the numbers of regular bipartite graphs as the size, and especially the valency, increases. Lehman matrices with $k > 1$ appear to be very rare, with the incidence matrices of projective planes being the only known infinite family and the adjacency matrices  of the Moore graphs giving a handful of sporadic examples. We note that this takes the \emph{usual} adjacency matrix, and then treats it as a clutter matrix.

In this paper, we consider in detail the structure of Lehman matrices by viewing them as bipartite graphs in the following manner. 
An $r$-regular $(0,1)$-matrix can be viewed as the \emph{bipartite adjacency matrix} of an $r$-regular bipartite graph and vice versa, so we say that a graph is a \emph{Lehman graph} if its bipartite adjacency matrix is a Lehman matrix.  We note that this differs from the $(\alpha,\omega)$-graph associated with a negative Lehman matrix, where each row of the matrix determines a clique of the graph, and every clique of the graph corresponds to a row of the matrix. Using the analogous graph for positive Lehman matrices immediately runs into problems because the graph usually has \emph{more}cliques than there are rows in the matrix. 

A matrix of order $n \times n$ corresponds to a bipartite graph of order $2n$ with $n$ black and $n$ white vertices. We primarily consider the case when $r=3$, where both the theoretical and computational tools give us most traction, and we call these graphs \emph{cubic} Lehman graphs. Figures~\ref{fig:leh22_01} and~\ref{fig:leh22_02} show all four cubic Lehman graphs on $22$ vertices. It is immediately apparent that they are qualitatively rather similar and in particular all of them seem to be very ``\emph{ladder-like}". The first graph of Figure~\ref{fig:leh22_01} actually \emph{is} the cubic M\"obius ladder of order $22$, while the others all appear to consist of \emph{ladder segments} of varying lengths connected together. We shall see that this is no accident and that a single construction technique involving the replacing of suitable pairs of edges by $6$-vertex ladder segments accounts for almost all of the known cubic Lehman graphs.

\begin{figure}[htb]
\begin{center}
\begin{tikzpicture}[scale=1.1]
\tikzstyle{rowvertex}=[circle,fill=white,draw=black ,inner sep = 0.6mm]
\tikzstyle{colvertex}=[circle,fill=black,draw=black ,inner sep = 0.6mm]
\foreach \x in {0,2,4,6,8,10,11} {
  \node [rowvertex] (v\x) at (106.36+360*\x/11:1.25cm) {};
  \node [colvertex] (w\x) at (106.36+360*\x/11:1.75cm) {};
  \draw (v\x)--(w\x);
}
\foreach \x in {1,3,5,7,9} {
  \node [colvertex] (v\x) at (106.36+360*\x/11:1.25cm) {};
  \node [rowvertex] (w\x) at (106.36+360*\x/11:1.75cm) {};
  \draw (v\x)--(w\x);
}
\draw (v0)--(v1)--(v2)--(v3)--(v4)--(v5)--(v6)--(v7)--(v8)--(v9)--(v10)--(w0)--(w1)--(w2)--(w3)--(w4)--(w5)--(w6)--(w7)--(w8)--(w9)--(w10)--(v0);
\end{tikzpicture}
\hspace{1cm}
\begin{tikzpicture}[scale=1.1]
\tikzstyle{rowvertex}=[circle,fill=white,draw=black ,inner sep = 0.6mm]
\tikzstyle{colvertex}=[circle,fill=black,draw=black ,inner sep = 0.6mm]
\foreach \x in {2,4,6,8} {
  \node [rowvertex] (v\x) at (106.36+360*\x/11:1.25cm) {};
  \node [colvertex] (w\x) at (106.36+360*\x/11:1.75cm) {};
  \draw (v\x)--(w\x);
}
\foreach \x in {3,5,7} {
  \node [colvertex] (v\x) at (106.36+360*\x/11:1.25cm) {};
  \node [rowvertex] (w\x) at (106.36+360*\x/11:1.75cm) {};
  \draw (v\x)--(w\x);
}

\node [rowvertex] (w9) at (-0.35,0.5) {};
\node [colvertex](w10) at (-0.35,0.9) {};
\node [rowvertex](v0) at (-0.35,1.3) {};
\node [colvertex](v1) at (-0.35,1.7) {};

\node [colvertex](v9) at (0.35,0.5) {};
\node [rowvertex](v10) at (0.35,0.9) {};
\node [colvertex](w0) at (0.35,1.3) {};
\node [rowvertex](w1) at (0.35,1.7) {};
\draw (v0)--(w0);
\draw (v1)--(w1);
\draw (v10)--(w10);
\draw (v9)--(w9);

\draw (v0)--(v1)--(v2)--(v3)--(v4)--(v5)--(v6)--(v7)--(v8)--(v9)--(v10)--(w0)--(w1);
\draw (w2)--(w3)--(w4)--(w5)--(w6)--(w7)--(w8);
\draw (w9)--(w10)--(v0);
\draw  (w2)--(w9);
\draw (w1)--(w8);
\end{tikzpicture}

\end{center}
\caption{Two $(11,3,4)$-Lehman graphs}
\label{fig:leh22_01}
\end{figure}
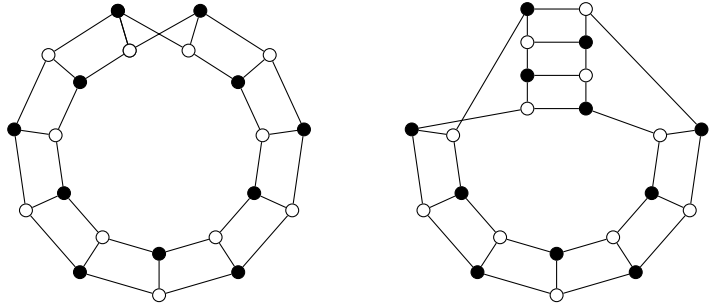

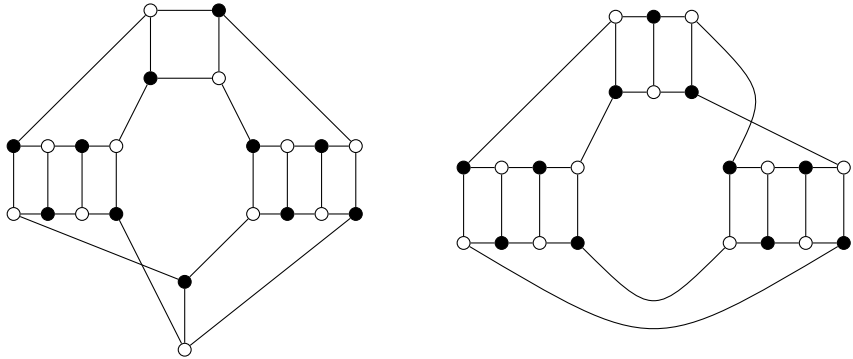
\begin{figure}[htb]
\begin{center}
\begin{tikzpicture} [scale = 0.9]
\tikzstyle{bvertex}=[circle,fill=black,draw=black ,inner sep = 0.6mm]
\tikzstyle{wvertex}=[circle,fill=white,draw=black ,inner sep = 0.6mm]
\node [wvertex] (v0) at (1,-0.5) {};
\node [bvertex]  (v1) at (1.5,-0.5) {};
\node [wvertex] (v2) at (2,-0.5) {};
\node [bvertex] (v3) at (2.5,-0.5) {};
\node [bvertex] (w0) at (1,0.5) {};
\node [wvertex]  (w1) at (1.5,0.5) {};
\node [bvertex] (w2) at (2,0.5) {};
\node [wvertex] (w3) at (2.5,0.5) {};

\node [bvertex] (x0) at (-1,-0.5) {};
\node [wvertex]  (x1) at (-1.5,-0.5) {};
\node [bvertex] (x2) at (-2,-0.5) {};
\node [wvertex] (x3) at (-2.5,-0.5) {};
\node [wvertex] (y0) at (-1,0.5) {};
\node [bvertex]  (y1) at (-1.5,0.5) {};
\node [wvertex] (y2) at (-2,0.5) {};
\node [bvertex] (y3) at (-2.5,0.5) {};

\node (b0) [bvertex] at (-0.5,1.5) {};
\node (b1) [wvertex] at (-0.5,2.5) {};
\node (b2) [wvertex] at (0.5,1.5) {};
\node (b3) [bvertex] at (0.5,2.5) {};

\node (a0) [bvertex] at (0,-1.5) {};
\node (a1) [wvertex] at (0,-2.5) {};

\draw (v0)--(v1)--(v2)--(v3)--(w3)--(w2)--(w1)--(w0)--(v0);
\draw (v1)--(w1);
\draw (v2)--(w2);

\draw (x0)--(x1)--(x2)--(x3)--(y3)--(y2)--(y1)--(y0)--(x0);
\draw (x1)--(y1);
\draw (x2)--(y2);

\draw (a0)--(a1);
\draw (b0)--(b1)--(b3)--(b2)--(b0);

\draw (w0)--(b2);
\draw (w3)--(b3);

\draw (v0)--(a0);
\draw (v3)--(a1);

\draw (a0)--(x3);
\draw (a1)--(x0);

\draw (b0)--(y0);
\draw (b1)--(y3);

\end{tikzpicture}
\hspace{1cm}
\begin{tikzpicture} [scale = 1]
\tikzstyle{bvertex}=[circle,fill=black,draw=black ,inner sep = 0.6mm]
\tikzstyle{wvertex}=[circle,fill=white,draw=black ,inner sep = 0.6mm]
\node [wvertex] (v0) at (1,-0.5) {};
\node [bvertex]  (v1) at (1.5,-0.5) {};
\node [wvertex] (v2) at (2,-0.5) {};
\node [bvertex] (v3) at (2.5,-0.5) {};
\node [bvertex] (w0) at (1,0.5) {};
\node [wvertex]  (w1) at (1.5,0.5) {};
\node [bvertex] (w2) at (2,0.5) {};
\node [wvertex] (w3) at (2.5,0.5) {};

\node [bvertex] (x0) at (-1,-0.5) {};
\node [wvertex]  (x1) at (-1.5,-0.5) {};
\node [bvertex] (x2) at (-2,-0.5) {};
\node [wvertex] (x3) at (-2.5,-0.5) {};
\node [wvertex] (y0) at (-1,0.5) {};
\node [bvertex]  (y1) at (-1.5,0.5) {};
\node [wvertex] (y2) at (-2,0.5) {};
\node [bvertex] (y3) at (-2.5,0.5) {};

\draw (v0)--(v1)--(v2)--(v3)--(w3)--(w2)--(w1)--(w0)--(v0);
\draw (v1)--(w1);
\draw (v2)--(w2);

\draw (x0)--(x1)--(x2)--(x3)--(y3)--(y2)--(y1)--(y0)--(x0);
\draw (x1)--(y1);
\draw (x2)--(y2);

\node [bvertex] (z0) at (-0.5,1.5) {};
\node [wvertex] (z1) at (0,1.5) {};
\node [bvertex] (z2) at (0.5,1.5) {};
\node [wvertex] (zz0) at (-0.5,2.5) {};
\node [bvertex] (zz1) at (0,2.5) {};
\node [wvertex] (zz2) at (0.5,2.5) {};

\draw (z0)--(z1)--(z2)--(zz2)--(zz1)--(zz0)--(z0);
\draw (z1)--(zz1);

\draw (zz2) .. controls (1.5,1.5)  .. (w0);
\draw (z2)--(w3);

\draw (z0)--(y0);
\draw (zz0)--(y3);

\draw (v0) .. controls (0,-1.5) .. (x0);
\draw (v3) .. controls (0,-2) .. (x3);
\end{tikzpicture}

\end{center}
\caption{The other two $(11,3,4)$-Lehman graphs}
\label{fig:leh22_02}
\end{figure}

More precisely, we show that if a cubic Lehman graph with $k= \pm 1$ contains a ladder segment with $3$ rungs, then it can be reduced to a smaller cubic Lehman graph with the same $k$ by removing the ladder segment and adding two edges to repair the regularity. We show that this process can be reversed, and used to construct huge numbers of cubic Lehman matrices with $k=\pm 1$ starting from the cubic planar ladder on $8$ vertices (for $k=-1$) or cubic M\"obius ladder on $10$ vertices (for $k=1$), and then repeatedly \emph{inserting} $3$-rung ladder segments. 

We repeat, verify, and extend L\"utolf \& Margot's computations, in the process discovering that their catalogue of $17 \times 17$ Lehman matrices omitted just one matrix --- the graph corresponding to this matrix is shown in Figure~\ref{lmmissing}. The sole omission is a Lehman graph of type $(17,3,6)$ that has no $4$-rung ladder segment, but that does have $3$-rung ladder segments. It is unclear as to how this graph/matrix was missed as the search described by L\"utolf \& Margot should certainly have constructed it at some stage.

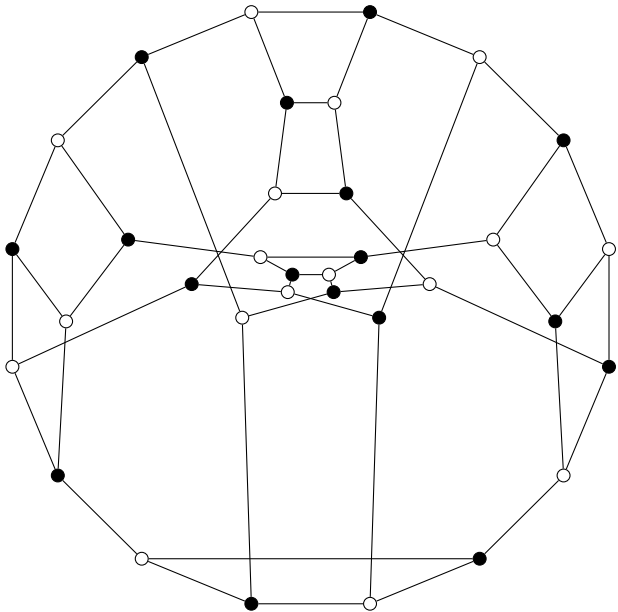
\begin{figure}[htb]
\begin{center}
\begin{tikzpicture}[scale=0.8,rotate=78.75]
\tikzstyle{rowv}=[circle, fill=black, draw=black, inner sep = 0.6mm]
\tikzstyle{colv}=[circle, fill=white, draw=black, inner sep = 0.6mm]    

\node [colv] (v0) at (-1.003558,3.899981) {};
\node [colv] (v1) at (-1.913417,4.619398) {};
\node [colv] (v2) at (1.913417,4.619398) {};
\node [colv] (v3) at (-4.619398,1.913417) {};
\node [colv] (v4) at (0.664717,0.973849) {};
\node [colv] (v5) at (0.182984,0.420792) {};
\node [colv] (v6) at (1.746812,0.945542) {};
\node [colv] (v7) at (-0.378127,1.072266) {};
\node [colv] (v8) at (4.619398,1.913417) {};
\node [colv] (v9) at (-4.619398,-1.913417) {};
\node [colv] (v10) at (-1.913417,-4.619398) {};
\node [colv] (v11) at (0.599792,-0.187200) {};
\node [colv] (v12) at (1.695873,-2.722678) {};
\node [colv] (v13) at (4.619398,-1.913417) {};
\node [colv] (v14) at (3.411659,0.280461) {};
\node [colv] (v15) at (0.768591,-1.841276) {};
\node [colv] (v16) at (1.913417,-4.619398) {};
\node [rowv] (v17) at (0.000000,5.000000) {};
\node [rowv] (v18) at (-3.535534,3.535534) {};
\node [rowv] (v19) at (0.524858,3.164409) {};
\node [rowv] (v20) at (0.005460,1.995244) {};
\node [rowv] (v21) at (3.535534,3.535534) {};
\node [rowv] (v22) at (-5.000000,0.000000) {};
\node [rowv] (v23) at (-3.535534,-3.535534) {};
\node [rowv] (v24) at (0.482498,0.402481) {};
\node [rowv] (v25) at (0.986794,-0.645343) {};
\node [rowv] (v26) at (0.060995,-1.135347) {};
\node [rowv] (v27) at (3.259290,1.046473) {};
\node [rowv] (v28) at (1.975687,-0.205091) {};
\node [rowv] (v29) at (0.330086,-0.318737) {};
\node [rowv] (v30) at (5.000000,0.000000) {};
\node [rowv] (v31) at (0.565291,-3.987158) {};
\node [rowv] (v32) at (0.000000,-5.000000) {};
\node [rowv] (v33) at (3.535534,-3.535534) {};
\draw (v0)--(v17);
\draw (v0)--(v18);
\draw (v0)--(v19);
\draw (v1)--(v17);
\draw (v1)--(v18);
\draw (v1)--(v20);
\draw (v2)--(v17);
\draw (v2)--(v19);
\draw (v2)--(v21);
\draw (v3)--(v18);
\draw (v3)--(v22);
\draw (v3)--(v23);
\draw (v4)--(v19);
\draw (v4)--(v24);
\draw (v4)--(v25);
\draw (v5)--(v20);
\draw (v5)--(v24);
\draw (v5)--(v26);
\draw (v6)--(v20);
\draw (v6)--(v27);
\draw (v6)--(v28);
\draw (v7)--(v21);
\draw (v7)--(v22);
\draw (v7)--(v29);
\draw (v8)--(v21);
\draw (v8)--(v27);
\draw (v8)--(v30);
\draw (v9)--(v22);
\draw (v9)--(v23);
\draw (v9)--(v26);
\draw (v10)--(v23);
\draw (v10)--(v31);
\draw (v10)--(v32);
\draw (v11)--(v24);
\draw (v11)--(v25);
\draw (v11)--(v29);
\draw (v12)--(v25);
\draw (v12)--(v31);
\draw (v12)--(v33);
\draw (v13)--(v26);
\draw (v13)--(v30);
\draw (v13)--(v33);
\draw (v14)--(v27);
\draw (v14)--(v28);
\draw (v14)--(v30);
\draw (v15)--(v28);
\draw (v15)--(v29);
\draw (v15)--(v32);
\draw (v16)--(v31);
\draw (v16)--(v32);
\draw (v16)--(v33);

\end{tikzpicture}
\end{center}
\caption{The ``missing'' Lehman graph on $34$ vertices}
\label{lmmissing}
\end{figure}

The computations also give us some sense of how many of the small cubic Lehman graphs arise from ladder insertions, simply by testing which of them have a $3$-rung ladder. Rather surprisingly, there are only \emph{two} cubic Lehman graphs with $k=1$ on up to $40$ vertices (corresponding to $20 \times 20$ matrices) that do \emph{not} have a $3$-rung ladder segment (Figure~\ref{fig:no3rung} shows the smaller example). The smallest cubic Lehman graph with $k=1$ is the M\"obius ladder on $10$ vertices, which is a $(5,3,2)$-Lehman graph. Therefore all except two cubic Lehman graphs (with $k=1$) on up to $40$ vertices arise from the M\"obius ladder on $10$ vertices by iterated ladder insertion.

The two exceptional cubic Lehman graphs having no $3$-rung ladder are also highly structured, in that their vertices can be partitioned into $4$-cycles. Motivated by this example, we describe a second reduction operation, which involves replacing the $4$-cycles with edges, thereby ``compressing'' a cubic Lehman graph with $k=1$ into a smaller cubic Lehman graph, but this time with $k=-1$. Unlike ladder insertion and its reverse, this construction is applicable to Lehman graphs of higher valency.  If an $r$-regular Lehman graph with $k= \pm 1$ can be partitioned into copies of the complete bipartite graph $K_{r-1,r-1}$ (which we denote \emph{bicliques}) then each biclique can be compressed to a single edge leaving a smaller $r$-regular graph with $k=\mp 1$ (respectively). We call this operation \emph{biclique compression} and we determine the circumstances under which it can be reversed (\emph{biclique expansion}) thereby producing a second construction technique for Lehman graphs. The square mni matrices discovered by Wang \cite{MR2739489} have the property that their vertices can be partitioned into copies of $K_{r-1,r-1}$ and so are instances of this construction.

\begin{figure}[htb]
\begin{center}
\begin{tikzpicture}[scale=1.25]
\tikzstyle{rowv}=[circle, fill=black, draw=black, inner sep = 0.6mm]
\tikzstyle{colv}=[circle, fill=white, draw=black, inner sep = 0.6mm]
\foreach \x in {0,2,4,6,8,10,12} {
  \node [rowv] (ir\x) at (\x*360/14:1.5cm) {};
  \node [rowv] (or\x) at (\x*360/14:2cm) {};
}
\foreach \x in {1,3,5,7,9,11,13} {
  \node[colv] (ic\x) at (\x*360/14:1.5cm) {};
  \node [colv] (oc\x) at (\x*360/14:2cm) {};
}

\foreach \x/\y in {0/1,2/3,4/5,6/7,8/9,10/11,12/13} {
 \draw (ir\x)--(ic\y);
 \draw (ir\x)--(oc\y);
 \draw (or\x)--(ic\y);
 \draw (or\x)--(oc\y);
}

\draw (oc1)--(or2);
\draw (oc3)--(or4);
\draw (oc5)--(or6);
\draw (oc7)--(or8);
\draw (oc9)--(or10);
\draw (oc11)--(or12);
\draw (oc13)--(or0);

\draw (ic1)--(ir4);
\draw (ic3)--(ir6);
\draw (ic5)--(ir8);
\draw (ic7)--(ir10);
\draw (ic9)--(ir12);
\draw (ic11)--(ir0);
\draw (ic13)--(ir2);

\end{tikzpicture}
\end{center}
\caption{A $(14,3,5)$-Lehman graph with no $3$-rung ladder segment}
\label{fig:no3rung}
\end{figure}
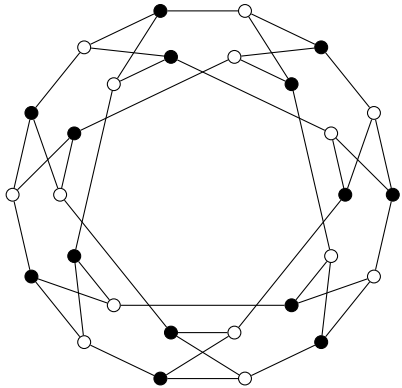

We note that it is always possible to insert enough $3$-rung ladder segments into a cubic Lehman graph with $k=1$ to ensure that the vertices of the resulting graph can be partitioned into $4$-cycles. Therefore every cubic Lehman graph with $k=1$ can be obtained from a negative cubic Lehman graph by a combination of biclique expansion followed by $3$-rung ladder reduction. In principle then, it suffices to characterise cubic negative Lehman graphs.

The paper is structured as follows: Section~\ref{prelims} contains all necessary background, definitions and notation for what follows. Section~\ref{ladders} gives a detailed analysis of the ladder reduction and insertion operations, while Section~\ref{completebipartite} does the same for biclique compression and expansion. Section~\ref{lehmancat} gives the results of a computer search for cubic Lehman graphs (with $k=\pm 1$) of order up to $20 \times 20$. Subsequent analysis of the data reveals that all of these Lehman graphs arise from the repeated application of our constructions (mostly ladder insertion) from a tiny number of base graphs.

Section~\ref{ppcore} addresses the question of when the square submatrix formed by the minimum weight rows of a minimally nonideal matrix is the point-line incidence matrix of a projective plane. It is known that the point-line incidence matrix of the Fano plane (with no added rows) is mni. We conjecture that no other mni matrices, square or otherwise, can be obtained by adding (zero or more) rows to the point-line incidence matrix of a projective plane. We prove that the conjecture holds if the projective plane is the Fano plane $\mathrm{PG}(2,2)$ or the ternary plane $\mathrm{PG}(2,3)$. 

\section{Preliminaries}
\label{prelims}

In this section we establish some the basic results regarding Lehman matrices that we will need later. Although many of these results apply equally for $k=-1$ and $k>0$, most previous authors have focused on one case or the other. Indeed, a substantial majority of the literature is devoted to the case $k=-1$ because, as outlined above, negative Lehman matrices are essentially equivalent to normalized $(\alpha,\omega)$-graphs about which a great deal is known.

% imperfect graphs via the concept of an $(\alpha,\omega)$-graph. A graph $G$ is defined to be an $(\alpha,\omega)$-graph if $|V(G)| = \alpha \omega + 1$ and, for every vertex $v$, the set $V(G)\backslash v$ can be partitioned into both $\omega$-cliques and $\alpha$-cocliques. The definition of an $(\alpha,\omega)$-graph and 

% (This ``partitioning property'' of the cliques and cocliques explains the alternative term of \emph{partitionable graph} for these graphs.) It can be shown that an $(\alpha,\omega)$-graph on $n$ vertices has exactly $n$ maximum cliques and $n$ maximum cocliques. Furthermore, if $A$ is a matrix whose rows are the characteristic vectors of the cliques, then $A$ is a negative Lehman matrix. Moreover, the matrix $B$ such that $AB^T = J-I$ has the characteristic vectors of the cocliques as its rows.

The following theorem is due to Bridges \& Ryser \cite{BR69} for $k>0$, although their proof works unchanged for $k=-1$.
\begin{theorem}
\label{BridgesRyser}
Let $A$ and $B$ be $n\times n$ non-negative integral matrices with $n>1$ such that $AB^T=J+kI$, where $k$ is in $\{-1,1,2,3,\ldots\}$. 
Then $B^TA=AB^T$ and there are integers $r$, $s$ such that $A$ is $r$-regular, $B$ is $s$-regular and $rs = n + k$.
\qed \end{theorem}

Thus, if $A$ is a Lehman matrix, then $A$ has constant row- and column-sum, and moreover, the matrix $B$ that satisfies $AB^T=J+kI$ is also a Lehman matrix. Given a non-singular $(0,1)$-matrix $A$ and an integer $k$, the only possible matrix that might form a Lehman pair with $A$ is
\[
B = \left( A^{-1} (J + k I) \right)^T
\]
so $(A,B)$ is a Lehman pair if and only if $B$ is a $(0,1)$-matrix. However a matrix can belong to two different Lehman pairs --- the bipartite adjacency matrix of the $6$-cycle is a $(3,2,1)$ Lehman matrix and also a $(3,2,2)$ Lehman matrix. 

\begin{corollary}
\label{cor1}
Let $A$ be a Lehman matrix satisfying $AB^T=J+kI$. Then $A^{T}B=J+kI$.
\end{corollary}

\begin{proof}
Assume that $AB^T=J+kI$. Theorem \ref{BridgesRyser} says that $B^TA=J+kI$. Therefore $J+kI=(J+kI)^{T}=(B^TA)^{T}=A^{T}B$.
\end{proof}

Let $G$ be a connected bipartite graph whose vertices are partitioned into two independent sets $G_B = \{b_1, b_2, \ldots, \}$ and $G_W = \{w_1, w_2, \ldots, \}$, which we refer to as the \emph{black vertices} and the \emph{white vertices} of $G$, respectively.  Then the \emph{bipartite adjacency matrix} of $G$ is the matrix $M$ with rows indexed by $G_B$ and columns indexed by $G_W$ where $M_{bw}=1$ if and only if $b$ is adjacent to $w$. Conversely, any $(0,1)$-matrix corresponds to a bipartite graph in the obvious fashion. If the matrix is a Lehman matrix, then its associated bipartite graph is regular. A bipartite graph is called a \emph{Lehman graph} if its bipartite adjacency matrix is a Lehman matrix. We note that for negative Lehman matrices, this bipartite representation is different from the representation as an $(\alpha,\omega)$-graph. In particular, given a negative bipartite Lehman graph $G$, the corresponding $(\alpha, \omega)$-graph $\Gamma$ has the white vertices of $G$ as its vertices, and a clique on the neighbours of each black vertex. Thus a white vertex of $G$ corresponds to a vertex of $\Gamma$, a black vertex of $G$ corresponds to an $\omega$-clique of $\Gamma$ and an edge of $G$ corresponds to a ``pointed clique'' (a clique with a distinguished vertex) of $\Gamma$. 
If $v$ is a vertex in a loopless graph, let $N(v)$ stand for its \emph{open neighbourhood}: that is, the set of vertices adjacent to $v$. The following proposition just reinterprets the definition of a Lehman matrix in graph-theoretical terms.

\begin{proposition}
\label{prop2}
Let $G$ be a regular bipartite graph with bipartition $\{G_B, G_W\}$. Then $G$ is a Lehman graph if and only if there is some integer $k \in \{-1, 1, 2, \ldots,\}$ such that for every black vertex $b$, there is a set $\Gamma_G(b) \subseteq G_W$ of white vertices such that, for all $b' \in G_B$
\[
|\Gamma_G(b) \cap N(b')| =
\begin{cases}
k+1, & b' = b;\\
1, & b'\not= b.
\end{cases} 
\] 
\end{proposition}

\begin{proof}
If $G$ is a Lehman graph with bipartite adjacency matrix $A$, then it belongs to a Lehman pair $(A,B)$ and we can take the rows of $B$ to be the (incidence vectors of the) sets $\Gamma_G(b)$. The Lehman condition for the matrix is then identical to the intersection conditions for the sets. The converse is very similar -- if a collection of suitable sets $\{\Gamma_G(b) \}_{b\in G_B}$ exists, then the matrix $B$ with the incidence vectors of these sets as its rows will form a Lehman pair with $A$.
\end{proof}

We will call the set $\mate{G}{b}$ the \emph{mate} of $b$, and will drop the subscript $G$ if the graph is uniquely determined by context. For both positive and negative Lehman graphs, each set $\mate{}{b}$ is a set of white vertices that \emph{dominates} every black vertex other than $b$ exactly once, while dominating $b$ exactly $k+1$ times. Figure~\ref{fig:desargues} shows the mate of a vertex in the \emph{Desargues graph}, which is a $(10,3,4)$-Lehman graph. As $k = 3 \times 4 - 10 = 2$, the Desargues graph is one of the rare Lehman graphs with $k \neq \pm 1$. The set of four circled white vertices dominates the marked black vertex $b$ three times, and all other black vertices exactly once each. 

If $G$ is a Lehman graph with bipartite adjacency matrix $A$, then the incidence vector $x$ of the mate of the vertex $b_i$ is the unique solution to $A x = \mathbf{1} + ke_i$ (where $e_i$ is the standard basis vector with a single $1$ in the $i$th 
position). Therefore if any black vertex of a regular bipartite graph $G$ has no mates or more than one mate, then $G$ is not a Lehman graph. If every black vertex of $G$ has \emph{at least} one mate, then every black vertex must have \emph{exactly} one mate, and the graph is a Lehman graph. By Corollary \ref{cor1}, we can swap the words ``black'' and ``white'' wherever they occur, and so every white vertex of a Lehman graph also has a unique mate. Our arguments in subsequent sections will largely be based around showing that a vertex in a candidate Lehman graph has too few, too many, or exactly the right number of mates.

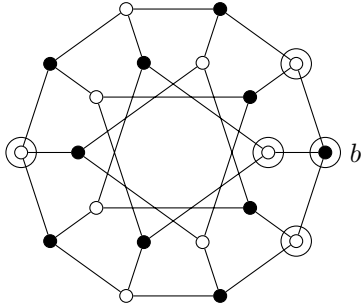
\begin{figure}[htb]
\begin{center}
\begin{tikzpicture}
\tikzstyle{rowv}=[circle, fill=black, draw=black, inner sep = 0.6mm]
\tikzstyle{colv}=[circle, fill=white, draw=black, inner sep = 0.6mm]
\foreach \x in {0,2,4,6,8} {
\node [rowv] (c\x) at (36*\x:2cm) {};
\node [colv] (r\x) at (36*\x:1.25cm) {};
\draw (r\x)--(c\x);
}

\foreach \x in {1,3,5,7,9} {
\node [rowv] (c\x) at (36*\x:1.25cm) {};
\node [colv] (r\x) at (36*\x:2cm) {};
\draw (r\x)--(c\x);
}

\draw (c0)--(r1)--(c2)--(r3)--(c4)--(r5)--(c6)--(r7)--(c8)--(r9)--(c0);
\draw (r0)--(c3)--(r6)--(c9)--(r2)--(c5)--(r8)--(c1)--(r4)--(c7)--(r0);

\draw (c0) circle (2mm);
\draw (r1) circle (2mm);
\draw (r0) circle (2mm);
\draw (r9) circle (2mm);
\draw (r5) circle (2mm);

\node at (0:2.4cm) {$b$};

\end{tikzpicture}
\end{center}
\caption{A mate in the unique $(10,3,4)$-Lehman graph}
\label{fig:desargues}
\end{figure}

\begin{proposition}
\label{prop4}
Let $G$ be a Lehman graph and suppose that $b$ is a black vertex, and $w$ is a white vertex. Then $w$ is in the mate of $b$ if and only if $b$ is in the mate of $w$.
\end{proposition}

\begin{proof}
Let $A$ be the Lehman matrix associated with $G$, and let $B$ be the matrix satisfying $AB^T=J+kI$. The rows of $B$ are the incidence vectors of the mates of the black vertices, and so $w$ is in the mate of $b$ if and only if $B_{bw} = 1$. Now $A^{T}B=J+kI$, by Corollary \ref{cor1}, and so the rows of $B^T$ are the incidence vectors of the mates of the white vertices. Thus $b$ is in the mate of $w$ if and only if $(B^T)_{wb} = 1$, which happens if and only if $B_{bw}=1$.
\end{proof}

The \emph{Hadamard product} of two square matrices $X$ and $Y$, written $X\circ Y$, is component-wise product of $X$ and $Y$; that is, $[X\circ Y]_{i,j}=[X]_{i,j}[Y]_{i,j}$ for all $i$ and $j$. Let $(A,B)$ be a Lehman pair satisfying $AB^T = J + k I$. Then $A\circ B$ is a $(k+1)$-regular matrix. Thus the bipartite graph corresponding to $A \circ B$ is $(k+1)$-regular. We call this the \emph{auxiliary graph} of $A$ (or $B$) and denote it $\operatorname{aux}(A)$. If $G$ is a cubic Lehman matrix with $k=1$, then its auxiliary graph is $2$-regular, and so the edges \emph{not} in the auxiliary graph form a perfect matching of $G$. We will call the edges of this distinguished perfect matching the \emph{rungs} of $G$. This terminology arises from the observation that if the graph actually is a ladder, either a cubic planar ladder or a cubic M\"obius ladder, then the rungs of the Lehman graph are actually the rungs of the ladder in the normal graph-theoretical sense. 

Figure~\ref{fig:aux} shows a $(14,3,5)$-Lehman graph, with the diagram on the left highlighting the auxiliary graph and the diagram on the right highlighting the rungs. 

\begin{figure}[htb]
\begin{center}
\begin{tikzpicture}[scale=1.25]
\tikzstyle{rowv}=[circle, fill=black, draw=black, inner sep = 0.6mm]
\tikzstyle{colv}=[circle, fill=white, draw=black, inner sep = 0.6mm]
\tikzstyle{ghost}=[dotted]

\foreach \x in {0,2,4,6,8} {
\node [rowv] (r\x) at (36*\x:2cm) {};
\node [colv] (c\x) at (36*\x:1cm) {};
}
\foreach \x in {1,3,5,7,9} {
\node [rowv] (r\x) at (36*\x:1cm) {};
\node [colv] (c\x) at (36*\x:2cm) {};
}

\node [colv] (w1) at (36:1.35cm) {};
\node [rowv] (b1) at (36:1.65cm) {};
\node [colv] (w2) at (2*36:1.65cm) {};
\node [rowv] (b2) at (2*36:1.35cm) {};

\node [colv] (w3) at (3*36:1.35cm) {};
\node [rowv] (b3) at (3*36:1.65cm) {};
\node [colv] (w4) at (4*36:1.65cm) {};
\node [rowv] (b4) at (4*36:1.35cm) {};

\draw (r0)--(c1)--(b1)--(w1)--(r1)--(c0)--(r9)--(c8)--(r7)--(c6)--(r5)--(c4)--(b4)--(w4)--(r4)--(c5)--(r6)--(c7)--(r8)--(c9)--(r0);
\draw (c2)--(b2)--(w2)--(r2)--(c3)--(b3)--(w3)--(r3)--(c2);

\foreach \x in {0,5,6,7,8,9} {
\draw [ghost] (r\x)--(c\x);
}

\draw [ghost](w1)--(b2);
\draw [ghost] (w2)--(b1);
\draw [ghost] (w3)--(b4);
\draw [ghost] (w4)--(b3);
\draw [ghost] (r1)--(c2);
\draw [ghost] (c1)--(r2);
\draw [ghost] (r3)--(c4);
\draw [ghost] (c3)--(r4);

\end{tikzpicture}
\hspace{1cm}
\begin{tikzpicture}[scale=1.25]
\tikzstyle{rowv}=[circle, fill=black, draw=black, inner sep = 0.6mm]
\tikzstyle{colv}=[circle, fill=white, draw=black, inner sep = 0.6mm]
\tikzstyle{ghost}=[dotted]
\foreach \x in {0,2,4,6,8} {
\node [rowv] (r\x) at (36*\x:2cm) {};
\node [colv] (c\x) at (36*\x:1cm) {};
}
\foreach \x in {1,3,5,7,9} {
\node [rowv] (r\x) at (36*\x:1cm) {};
\node [colv] (c\x) at (36*\x:2cm) {};
}

\node [colv] (w1) at (36:1.35cm) {};
\node [rowv] (b1) at (36:1.65cm) {};
\node [colv] (w2) at (2*36:1.65cm) {};
\node [rowv] (b2) at (2*36:1.35cm) {};

\node [colv] (w3) at (3*36:1.35cm) {};
\node [rowv] (b3) at (3*36:1.65cm) {};
\node [colv] (w4) at (4*36:1.65cm) {};
\node [rowv] (b4) at (4*36:1.35cm) {};

\foreach \x in {0,5,6,7,8,9} {
\draw (r\x)--(c\x);
}

\draw (w1)--(b2);
\draw (w2)--(b1);
\draw (w3)--(b4);
\draw (w4)--(b3);
\draw (r1)--(c2);
\draw (c1)--(r2);
\draw (r3)--(c4);
\draw (c3)--(r4);

\draw [ghost] (r0)--(c1)--(b1)--(w1)--(r1)--(c0)--(r9)--(c8)--(r7)--(c6)--(r5)--(c4)--(b4)--(w4)--(r4)--(c5)--(r6)--(c7)--(r8)--(c9)--(r0);
\draw [ghost] (c2)--(b2)--(w2)--(r2)--(c3)--(b3)--(w3)--(r3)--(c2);

\end{tikzpicture}
\end{center}
\caption{Auxiliary graph and rungs of a $(14,3,5)$-Lehman graph}
\label{fig:aux}
\end{figure}
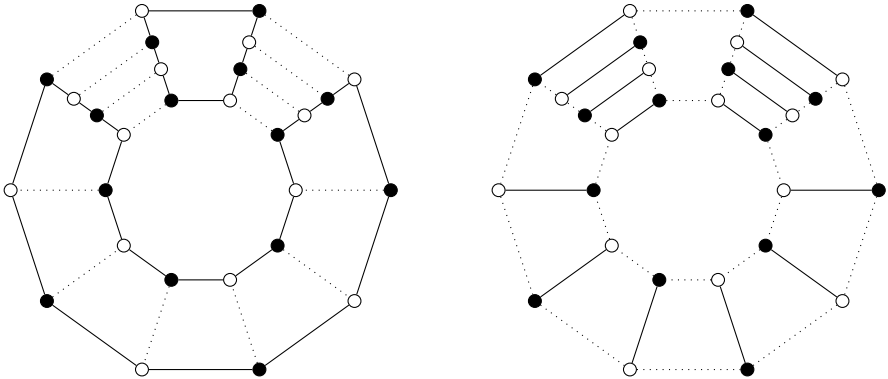

\section{Ladder reduction and insertion}
\label{ladders}

In this section, we describe the first of the two ways in which certain Lehman graphs can be reduced to smaller Lehman graphs, and when this operation can be reversed. This operation applies only to cubic Lehman graphs with $k \in \{-1,1\}$.

A \emph{$3$-rung ladder segment} is a $6$-vertex induced subgraph isomorphic to the graph obtained from the cube $Q_3$ by deleting two adjacent vertices, along with their incident edges. Except for two small base cases, we show that it is always possible to delete a $3$-rung ladder segment from a cubic Lehman graph and then add two edges to repair the regularity in such a way that the resulting graph is also a Lehman graph. This reduction produces Lehman graphs with six fewer vertices than the original graph, but with the same ``sign'' (i.e.\ positive or negative). 

Boros, Gurvich and Hougardy \cite{MR1936943} found a reduction operation that produces an $(\alpha-1,\omega)$-graph by removing the vertices of a \emph{critical clique} from an $(\alpha,\omega)$-graph and adding new edges to restore partitionability. Although the language used is quite different, it is straightforward to see that their clique-reduction and our ladder-reduction coincide when $\omega=3$ and $k=-1$.

\subsection{Ladder Reduction For Cubic Lehman Graphs}
\label{reduction}

\newcommand\compress[2]{#1\mathord{\downarrow}#2}
\newcommand\gdl{\compress{G}{L}}
 
We define a reduction operation, which we denote \emph{$3$-rung ladder reduction}, that preserves the property of being a cubic Lehman graph. More precisely, suppose that $L$ is a $3$-rung ladder segment in a cubic graph $G$, that there are four distinct vertices $\{w_L, b_L, w_R, b_R\}$ that are adjacent to, but outside, $L$ (as shown in Figure~\ref{fig:3rung}), and that the pairs $(b_L, w_R)$ and $(w_L, b_R)$ are not edges of $G$. The dotted lines in the figure represent edges that may or may not be present. Then the \emph{$3$-rung ladder reduction} of $G$ with respect to $L$ is the graph $\gdl$ obtained from $G$ by deleting the six vertices of $L$ and then restoring $3$-regularity by adding the edges $(b_L, w_R)$ and $(b_L, w_R)$. 

\begin{figure}[htb]
\begin{center}
\begin{tikzpicture}
\tikzstyle{rowvertex}=[circle,fill=black,draw=black ,inner sep = 0.8mm]
\tikzstyle{colvertex}=[circle,fill=white,draw=black ,inner sep = 0.8mm]
\tikzstyle{grayvertex}=[circle,fill=lightgray,draw=black ,inner sep = 0.8mm]
\node [colvertex] (c0) at (0.2,0.2) {};
\node [colvertex] (c1) at (1,2) {};
\node [colvertex] (c2) at (2,2) {};
\node [colvertex] (c3) at (3,1) {};
\node [colvertex] (c4) at (3.8,2.8) {};
\node [rowvertex] (r0) at (0.2,2.8) {};
\node [rowvertex] (r1) at (1,1) {};
\node [rowvertex] (r2) at (2,1) {};
\node [rowvertex] (r3) at (3,2) {};
\node [rowvertex] (r4) at (3.8,0.2) {};
\draw (c1)--(r1);
\draw (c2)--(r2);
\draw (c3)--(r3);
\draw (c0)--(r1)--(c2)--(r3)--(c4);
\draw (r0)--(c1)--(r2)--(c3)--(r4);
\node [above] at (c1.north) {$w_0$};
\node [above] at (c2.north) {$w_1$};
\node [below] at (c3.south) {$w_2$};
\node [below] at (c0.south) {$w_L$};
\node [above] at (c4.north) {$w_R$};

\node [below] at (r1.south) {$b_0$};
\node [below] at (r2.south) {$b_1$};
\node [above] at (r0.north) {$b_L$};
\node [above] at (r3.north) {$b_2$};
\node [below] at (r4.south) {$b_R$};
\draw [dashed] (r0)--(c0);
\draw [dashed] (r4)--(c4);

\draw [->] (4.5,1.5)--(5.5,1.5);

\pgftransformxshift{6cm}

\node [colvertex] (c0) at (0.2,0.2) {};
\node [colvertex] (c4) at (3.8,2.8) {};
\node [rowvertex] (r0) at (0.2,2.8) {};
\node [rowvertex] (r4) at (3.8,0.2) {};
\node [below] at (c0.south) {$w_L$};
\node [above] at (c4.north) {$w_R$};
\node [above] at (r0.north) {$b_L$};
\node [below] at (r4.south) {$b_R$};
\draw (c0)--(r4);
\draw (c4)--(r0);
\draw [dashed] (r0)--(c0);
\draw [dashed] (r4)--(c4);

\end{tikzpicture}
\end{center}
\caption{A $3$-rung ladder reduction}
\label{fig:3rung}
\end{figure}
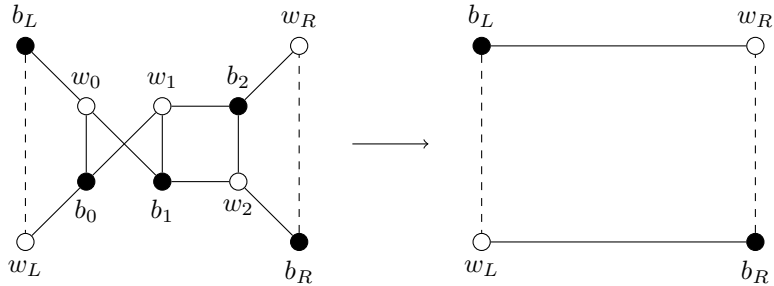

The constraints on $L$ and its vertices of attachment are simply the conditions required to ensure that $\gdl$ is actually cubic. These conditions are necessary because there are two small Lehman graphs that each contain a $3$-rung ladder, but which cannot be reduced with this operation. These graphs, illustrated in Figure~\ref{fig:smallgraphs} are the cubic planar ladder on $8$ vertices (i.e. the cube) and the cubic M\"obius ladder on $10$ vertices. The cube cannot be reduced because the vertices $\{b_L, w_L, b_R, w_R\}$ are not distinct while the $10$-vertex ladder cannot be reduced because $b_L$ is already adjacent to $w_R$ (similarly for $w_L$ and $b_R$). 

\begin{figure}[htb]
\begin{center}
\begin{tikzpicture}[scale=0.85]
\tikzstyle{rowvertex}=[circle,fill=black,draw=black ,inner sep = 0.8mm]
\tikzstyle{colvertex}=[circle,fill=white,draw=black ,inner sep = 0.8mm]

\node [rowvertex] (b0) at (45:1cm) {};
\node [rowvertex] (b1) at (135:2cm) {};
\node [rowvertex] (b2) at (225:1cm) {};
\node [rowvertex] (b3) at (315:2cm) {};

\node [colvertex] (w0) at (45:2cm) {};
\node [colvertex] (w1) at (135:1cm) {};
\node [colvertex] (w2) at (225:2cm) {};
\node [colvertex] (w3) at (315:1cm) {};

\foreach \x in {0,1,2,3}{
\draw (b\x)--(w\x);
}

\draw (b0)--(w1)--(b2)--(w3)--(b0);
\draw (b1)--(w2)--(b3)--(w0)--(b1);

\pgftransformxshift{6cm}

\node [rowvertex] (b0) at (306:1cm) {};
\node [rowvertex](b1) at (18:2cm) {};
\node [rowvertex](b2) at (90:1cm) {};
\node [rowvertex](b3) at (162:2cm) {};
\node [rowvertex](b4) at (234:1cm) {};

\node [colvertex] (w0) at (306:2cm) {};
\node [colvertex](w1) at (18:1cm) {};
\node [colvertex](w2) at (90:2cm) {};
\node [colvertex](w3) at (162:1cm) {};
\node [colvertex](w4) at (234:2cm) {};

\foreach \x in {0,1,2,3,4}{
 \draw (b\x)--(w\x);
}

\draw (w3)--(b4)--(w0)--(b1)--(w2)--(b3)--(w4)--(b0)--(w1)--(b2)--(w3);

\end{tikzpicture}
\caption{Lehman graphs of type $(4,3,1)$ and $(5,3,2)$}
\label{fig:smallgraphs}
\end{center}
\end{figure}
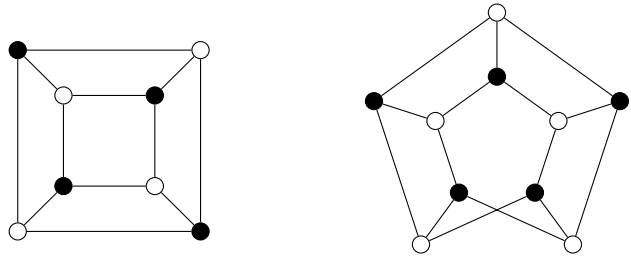

Our first lemma shows that for all larger cubic Lehman graphs, any $3$-rung ladder segment automatically meets these additional constraints, and so is suitable for ladder reduction.

\begin{lemma}
\label{forbiddenconfigs}
Suppose that $G$ is a Lehman graph of type $(n,3,s)$ where $k = 3 s - n \in \{1,-1\}$, and that $G$ contains a $3$-rung ladder $L$. If $k=-1$ and $s > 1$, or if $k=1$ and $s > 2$, then the vertices $\{b_L, w_L, b_R, w_R\}$ are distinct, $w_L$ is not adjacent to $b_R$, and $w_R$ is not adjacent to $b_L$. 
\end{lemma}

\begin{proof}
First consider the case where $k=1$ and suppose, for a contradiction, that $G$ contains a \ladseg{3} $L$ where $b_L = b_R$, as shown in Figure~\ref{fig:bl_and_br_distinct}. (Here $w_L$ and $w_R$ may be the same or distinct.) The mate $\Gamma(b_1)$ contains two vertices from $\{w_0, w_1, w_2\}$. However any pair of vertices from that set dominates two vertices twice, and so cannot be contained in the mate of any black vertex. Therefore the four vertices $\{b_L, w_L, b_R, w_R\}$ are indeed distinct. 

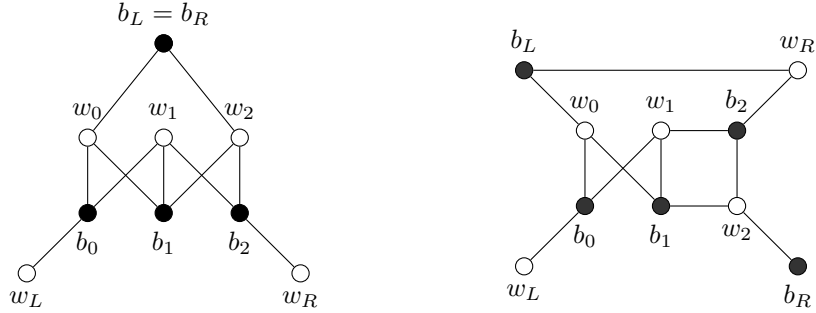
\begin{figure}[htb]
\begin{center}
\begin{tikzpicture}
\tikzstyle{rowvertex}=[circle,fill=black,draw=black ,inner sep = 0.8mm]
\tikzstyle{colvertex}=[circle,fill=white,draw=black ,inner sep = 0.8mm]
\tikzstyle{grayvertex}=[circle,fill=lightgray,draw=black ,inner sep = 0.8mm]
\node [colvertex] (wl) at (0.2,0.2) {};
\node [colvertex] (w0) at (1,2) {};
\node [colvertex] (w1) at (2,2) {};
\node [colvertex] (w2) at (3,2) {};
\node [colvertex] (wr) at (3.8,0.2) {};
\node [rowvertex] (bl) at (2,3.25) {};
\node [rowvertex] (b0) at (1,1) {};
\node [rowvertex] (b1) at (2,1) {};
\node [rowvertex] (b2) at (3,1) {};
\node [rowvertex] (br) at (2,3.25) {};
\draw (w0)--(b0);
\draw (w1)--(b1);
\draw (w2)--(b2);
\draw (wl)--(b0)--(w1)--(b2)--(wr);
\draw (bl)--(w0)--(b1)--(w2)--(br);
\node [above] at (w0.north) {$w_0$};
\node [above] at (w1.north) {$w_1$};
\node [above] at (w2.north) {$w_2$};
\node [below] at (wl.south) {$w_L$};
\node [below] at (wr.south) {$w_R$};
\node [below] at (b0.south) {$b_0$};
\node [below] at (b1.south) {$b_1$};
\node [above] at (bl.north) {$b_L = b_R$};
\node [below] at (b2.south) {$b_2$};
\end{tikzpicture}
\hspace{2cm}
\begin{tikzpicture}
\tikzstyle{rowvertex}=[circle,fill=black!80!white,draw=black ,inner sep = 0.8mm]
\tikzstyle{colvertex}=[circle,fill=white,draw=black ,inner sep = 0.8mm]
\tikzstyle{grayvertex}=[circle,fill=lightgray,draw=black ,inner sep = 0.8mm]
\node [colvertex] (wl) at (0.2,0.2) {};
\node [colvertex] (w0) at (1,2) {};
\node [colvertex] (w1) at (2,2) {};
\node [colvertex] (w2) at (3,1) {};
\node [colvertex] (wr) at (3.8,2.8) {};
\node [rowvertex] (bl) at (0.2,2.8) {};
\node [rowvertex] (b0) at (1,1) {};
\node [rowvertex] (b1) at (2,1) {};
\node [rowvertex] (b2) at (3,2) {};
\node [rowvertex] (br) at (3.8,0.2) {};
\draw (w0)--(b0);
\draw (w1)--(b1);
\draw (w2)--(b2);
\draw (wl)--(b0)--(w1)--(b2)--(wr);
\draw (bl)--(w0)--(b1)--(w2)--(br);
\node [above] at (w0.north) {$w_0$};
\node [above] at (w1.north) {$w_1$};
\node [below] at (w2.south) {$w_2$};
\node [below] at (wl.south) {$w_L$};
\node [above] at (wr.north) {$w_R$};

\node [below] at (b0.south) {$b_0$};
\node [below] at (b1.south) {$b_1$};
\node [above] at (bl.north) {$b_L$};
\node [above] at (b2.north) {$b_2$};
\node [below] at (br.south) {$b_R$};
\draw (bl)--(wr);
\end{tikzpicture}
\end{center}
\caption{Configurations in a cubic Lehman graph}
\label{fig:bl_and_br_distinct}
\end{figure}

Next we will show that $b_L$ is not adjacent to $w_R$. Again we proceed by contradiction starting from the second diagram of Figure~\ref{fig:bl_and_br_distinct}. Suppose that $b \not= b_1$ and that $\Gamma(b)$ contains $w_0$. Then $w_1, w_2 \notin \Gamma(b)$ and so to dominate $b_2$ it is forced that $w_R \in \Gamma(b)$ and as $b_L$ is now twice-dominated, it follows that $b = b_L$. Thus $w_0$ is in at most two mates, hence it follows that $s \leq 2$, contradicting our assumptions. (We note that this configuration can occur when $s=2$, as it is a subgraph of the cubic M\"obius ladder on $10$ vertices, which is a Lehman graph of type $(5,3,2)$.)

Next consider the case where $k=-1$ and suppose, again for a contradiction that $G$ contains the configuration of Figure~\ref{fig:bl_and_br_distinct}. Now we will consider the mates of the \emph{white} vertices, which are sets of black vertices dominating each white vertex bar one exactly once each. The exceptional white vertex is not dominated at all. If a mate contains one vertex from $\{b_0, b_1, b_2\}$ then it cannot contain $b_L$ for then we would have a vertex dominated twice. As $w_1$ is the only vertex whose mate contains no vertices from $\{b_0,b_1,b_2\}$ this means $b_L$ is only in one mate. Therefore $s=1$, $w_L = w_R$ and the entire graph is the cube. 

Next we will show that $b_L$ is not adjacent to $w_R$ and again proceed by contradiction starting from the second diagram of 
Figure~\ref{fig:bl_and_br_distinct}. Let $b$ be an arbitrary black vertex and suppose that $w_0 \in \Gamma(b)$. Then $w_1,w_2,w_R \notin \Gamma(b)$ because that would cause either $b_0$, $b_1$ or $b_L$ to be twice dominated. As $\Gamma(b)$ misses the entire neighbourhood of $b_2$, it follows that $b=b_2$ and $s=1$, contradicting our assumptions.
\end{proof}

The main purpose of Lemma~\ref{forbiddenconfigs} is to ensure that for any Lehman graph other than the cube and the cubic M\"obius ladder on $10$ vertices, a $3$-rung ladder reduction will at least give a cubic graph. Next we show that reducing $3$-rung ladders also preserves the Lehman property. 

\begin{proposition} 
Suppose that $G$ is a Lehman graph of type $(n,3,s)$ where $k = 3 s - n \in \{1,-1\}$. Furthermore, assume that $s > 2$ if $k=1$ and $s>1$ if $k=-1$. If $G$ contains a $3$-rung ladder $L$, then $\compress{G}{L}$ is a Lehman graph of type $(n-3,3,s-1)$.
\end{proposition}

\begin{proof}
From Lemma~\ref{forbiddenconfigs}, any $3$-rung ladder $L$ in $G$ has the form depicted in Figure~\ref{fig:3rung} where $\{w_L,w_R,b_L,b_R\}$ are distinct vertices and $(w_L, b_L)$ and $(w_R,b_R)$ are the only possible edges between vertices in this set. Therefore $\compress{G}{L}$ is at least a cubic graph. 

To show that it is a Lehman graph, we show that each black vertex $b \in V(\compress{G}{L})$ has a mate, and in fact we claim that 
\[
\mate{\compress{G}{L}}{b} = \mate{G}{b} \backslash \{w_0, w_1, w_2\}.
\]
In other words, take the mates of all the vertices in $G$, completely throw away the mates of $b_0$, $b_1$ and $b_2$ and then just delete $w_0$, $w_1$ and $w_2$ from the remainder.

First we show that in both the positive and negative cases, the mate in $G$ of any black vertex $b \notin \{b_0,b_1,b_2\}$ contains the vertex $w_0$ if and only if it contains $w_R$. This follows because if $w_0 \in \mate{G}{b}$ then $w_1$, $w_2 \notin \mate{G}{b}$ (as this would result in either $b_0$ or $b_1$ being twice-dominated). As $b \not= b_2$, the mate of $b$ must dominate $b_2$ exactly once and so $w_R \in \mate{G}{b}$. For the converse, note that if $w_R \in \mate{G}{b}$, then $w_1$, $w_2 \notin \mate{G}{b}$ because that would cause $b_2$ to be twice-dominated. In order to dominate $b_1$, it must be the case that $w_0 \in \mate{G}{b}$. Symmetrically, $w_2 \in \mate{G}{b}$ if and only if $w_L \in \mate{G}{b}$.

Next we will use this fact to show that for any $b \notin \{b_0, b_1, b_2\}$, the set $\mate{G}{b} \backslash \{w_0, w_1, w_2\}$ is a mate for the vertex $b$ in $\compress{G}{L}$. First observe that $\mate{G}{b}$ contains exactly one vertex in $\{w_0, w_1, w_2\}$. This follows because $b_1$ must be dominated at least once, and only $\mate{G}{b}$ can contain either zero or two of these vertices (in the negative, positive case respectively). If $\mate{G}{b}$ contains $w_0$ (the case $w_2$ is equivalent), then it also contains $w_R$, and although $b_L$ is no longer dominated in $\compress{G}{L}$ by the deleted vertex $w_0$, it is now dominated by $w_R$ via the added edge. Hence $\mate{G}{b} \backslash \{w_0, w_1, w_2\}$ dominates the black vertices in $\gdl$ exactly the same number of times as $\mate{G}{b}$ dominates the same vertices in $G$. If $\mate{G}{b}$ contains $w_1$, then it does not contain any of $w_0$, $w_2$, $w_L$, or $w_R$, for otherwise $b_{0}$ or $b_{2}$ is dominated twice, and $b\ne b_{0},b_{2}$. Therefore every black vertex of $\gdl$ is dominated by $\mate{G}{b} \backslash \{w_0, w_1, w_2\}$ exactly the same number of times as the same vertex was dominated by $\mate{G}{b}$ in $G$.

So we have shown that each black vertex has a mate, and therefore there is a solution to the defining (matrix) equation of a Lehman matrix and so $\compress{G}{L}$ is a Lehman graph. As $\compress{G}{L}$ has six fewer vertices than $G$, and each mate in $\gdl$ has cardinality $s-1$, we see that its parameters are $(n-3,3,s-1)$.
\end{proof}

The consequence of this result is that any cubic Lehman graph on more than $10$ vertices with $k=1$ and with a $3$-rung ladder segment can be reduced to a smaller cubic Lehman graph with $k=1$. If the reduced graph itself has a $3$-rung ladder segment, then the process can be iterated, ending with either the cubic M\"obius ladder on $10$ vertices or an ``irreducible'' Lehman graph with no $3$-rung ladder segment. Our exhaustive computer search for cubic Lehman graphs with $k=1$ on up to $40$ vertices show that there are just two irreducible graphs in this range, the smaller of which is shown in Figure~\ref{fig:no3rung}. 

\subsection{Ladder Insertion For Cubic Lehman Graphs}

Now we consider when the reverse operation of a $3$-rung ladder reduction can be performed. The reverse operation consists of removing a non-incident pair of edges $e = (w_L,b_R)$ and $f=(w_R,b_L)$, adding a new $3$-rung ladder segment $L$ (again labelled as in Figure~\ref{fig:3rung}), and finally adding the edges $(w_0,b_L)$, $(w_2,b_R)$ $(w_L,b_0)$ and $(w_2,b_R)$. We denote the resulting graph by $G\mathord{\uparrow}\{e,f\}$ and call the pair of edges \emph{expandable} if $G\mathord{\uparrow}\{e,f\}$ is a Lehman graph.  

For a cubic Lehman graph with $k=1$, the partition of the edge set into the $2$-regular auxiliary graph and its complementary perfect matching play a major role in determining when a $3$-ladder expansion yields a larger Lehman graph. The next proposition gives some simple conditions sufficient to guarantee that $\{e,f\}$ is an expandable pair of edges. In this proof, we will frequently need to refer to the rows, columns and individual entries of several different matrices, so to avoid cramped subscripts we temporarily use more prominent notation. More precisely, if a matrix $X$ has rows indexed by black vertices and columns by white vertices then $X(b,w)$ will refer to the $(b,w)$-entry of the matrix, $X(b)$ will refer to the \emph{row} indexed by $b$ and $X(w)$ will refer to the \emph{column} indexed by $w$.

\begin{proposition}\label{prop:positiveexpansion}
Let $G$ be a $(3s-1,3,s)$-Lehman graph and let $e = (w_L, b_R)$ and $f = (w_R, b_L)$ be non-incident edges of $G$. If $e$ and $f$ are in the auxiliary graph of $G$, the mates of $b_L$ and $b_R$ are disjoint, and the mates of $w_L$ and $w_R$ are disjoint, then $G\mathord{\uparrow}\{e,f\}$ is a $(3(s+1)-1,3,s+1)$-Lehman graph.
\end{proposition}

\begin{proof}
Suppose that $A$ and $A'$ are the bipartite adjacency matrices of $G$ and $G\mathord{\uparrow}\{e,f\}$ respectively. Simply translating the insertion operation into matrix terms, we see that $A'$ is obtained from $A$ by adding three additional rows and columns, and has the block form
\[
A' = 
\begin{bmatrix}
A'_{11} & A'_{12} \\
A'_{21} & A'_{22}
\end{bmatrix},
\]
where $A'_{11}$ is equal to $A$ except that $A'_{11} (b_L,w_R)$ and $A'_{11} (b_R,w_L)$ have been changed from $1$ to $0$ (the dashed arrows in Figures~\ref{fig:pos3ladder} and \ref{fig:neg3ladder} indicate that the original $1$ has been ``moved'' leaving behind a $0$ entry). Then $A'_{12}$ is a $(3s-1) \times 3$ matrix with just two non-zero entries in the $(b_L, w_0)$ and $(b_R,w_2)$ positions, and $A'_{21}$ is a $3 \times (3s-1)$ matrix with just two non-zero entries in the $(b_0, w_L)$ and $(b_2,w_R)$ positions. Finally $A'_{22}$ is the $3 \times 3$ matrix shown in Figure~\ref{fig:pos3ladder}.

\begin{figure}[htb]
\begin{center}
\begin{tikzpicture}[scale=0.45]

\node at (1,10) {$A$};
\draw (0,0) rectangle (11,11);
\node (wL) at (2.5,11.5){$w_L$};
\node (wR) at (4.5,11.5){$w_R$};
\node (w0) at (8.5,11.5){$w_0$};
\node (w1) at (9.5,11.5){$w_1$};
\node (w2) at (10.5,11.5){$w_2$};

\node (bR) at (-0.5,5.5){$b_R$};
\node (bL) at (-0.5,8.5){$b_L$};
\node (b0) at (-0.5,2.5){$b_0$};
\node (b1) at (-0.5,1.5){$b_1$};
\node (b2) at (-0.5,0.5){$b_2$};

\node [inner sep = 1pt] (wlbr) at (2.5,5.5) {\small $0$};
\node [inner sep = 1pt] (wrbl) at (4.5,8.5) {\small $0$};

\node [inner sep = 1pt] (b0wl) at (2.5,2.5) {\small $1$};
\node [inner sep = 1pt] (b2wr) at (4.5,0.5) {\small $1$};
\node [inner sep = 1pt] (blw0) at (8.5,8.5) {\small $1$};
\node [inner sep = 1pt] (brw2) at (10.5,5.5) {\small $1$};

\draw [dashed, ->] (wlbr) -- (b0wl);
\draw [dashed, ->] (wrbl) -- (b2wr);
\draw [dashed, ->] (wlbr) -- (brw2);
\draw [dashed, ->] (wrbl) -- (blw0);

\draw [thick, black] (8,0) -- (8,11);
\draw [thick, black] (0,3) -- (11,3);

\node at (8.5,0.5) {\small $0$};
\node at (9.5,0.5) {\small $1$};
\node at (10.5,0.5) {\small $1$};
\node at (8.5,1.5) {\small $1$};
\node at (9.5,1.5) {\small $1$};
\node at (10.5,1.5) {\small $1$};
\node at (8.5,2.5) {\small$1$};
\node at (9.5,2.5) {\small $1$};
\node at (10.5,2.5) {\small $0$};

\end{tikzpicture}
\hspace{0.5cm}
\begin{tikzpicture}[scale=0.45]
\node at (1,10) {$B$};
\draw (0,0) rectangle (11,11);
\node (wL) at (2.5,11.5){$w_L$};
\node (wR) at (4.5,11.5){$w_R$};
\node (w0) at (8.5,11.5){$w_0$};
\node (w1) at (9.5,11.5){$w_1$};
\node (w2) at (10.5,11.5){$w_2$};

\node (bR) at (-0.5,5.5){$b_R$};
\node (bL) at (-0.5,8.5){$b_L$};
\node (b0) at (-0.5,2.5){$b_0$};
\node (b1) at (-0.5,1.5){$b_1$};
\node (b2) at (-0.5,0.5){$b_2$};

\draw [thick, black] (8,0) -- (8,11);
\draw [thick, black] (0,3) -- (11,3);

\node at (2.5,5.5) {\small $1$};
\node at (4.5,8.5) {\small $1$};

\node at (2.5,8.5) {\small $0$};
\node at (4.5,5.5) {\small $0$};

\node at (8.5,0.5) {\small $0$};
\node at (9.5,0.5) {\small  $1$};
\node at (10.5,0.5) {\small $0$};
\node at (8.5,1.5) {\small $1$};
\node at (9.5,1.5) {\small $0$};
\node at (10.5,1.5) {\small $1$};
\node at (8.5,2.5) {\small $0$};
\node at (9.5,2.5) {\small $1$};
\node at (10.5,2.5) {\small $0$};

\node [rotate = -90] (BwR) at (8.5,7) {\small $B(w_R)$};
\node [rotate = -90] (jw) at (9.5,7) {\small $\allones -B(w_R)-B(w_L)$};
\node [rotate = -90] (BwL) at (10.5,7) {\small $B(w_L)$};

\draw [<-] (8.5,3.1)--(BwR);
\draw [->] (BwR)--(8.5,10.9);

\draw [<-] (10.5,3.1)--(BwL);
\draw [->] (BwL)--(10.5,10.9);

\node (Bbl) at (4,0.5) {\small $B(b_L)$};
\draw [<-] (0.1,0.5)--(Bbl);
\draw [<-] (7.9,0.5)--(Bbl);

\node (jbrbl) at (4,1.5) {\small $\allones -B(b_R)-B(b_L)$};

\node (Bbr) at (4,2.5) {\small $B(b_R)$};
\draw [<-] (0.1,2.5)--(Bbr);
\draw [<-] (7.9,2.5)--(Bbr);

\end{tikzpicture}
\caption{$A'$ and $B'$ after a $3$-ladder insertion in a positive Lehman graph}
\label{fig:pos3ladder}
\end{center}
\end{figure}
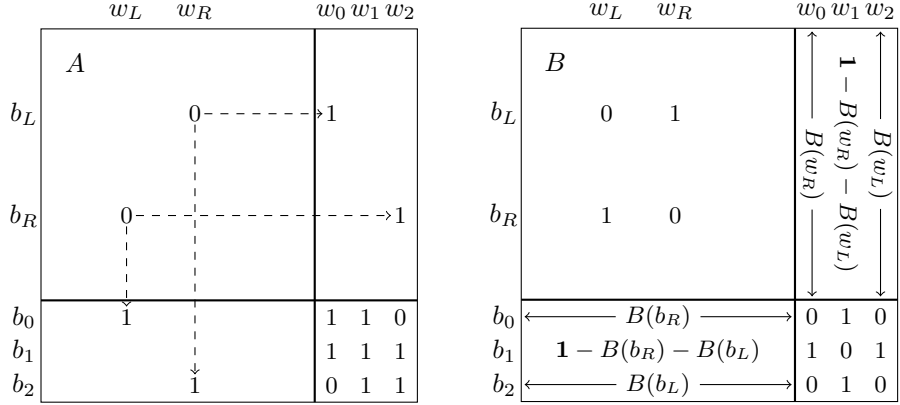

Now we will define a matrix $B'$, and then prove that $A'B' = J + I$ as required. This is constructed from $B$ (the partner of $A$ in the Lehman pair $(A,B)$) by adding three additional rows and columns, and is illustrated in the second diagram of Figure~\ref{fig:pos3ladder}. The upper left submatrix of $B'$ is simply equal to $B$, while the new rows and columns of $B'$ are defined in the following way: $B'(b_0)$ is obtained by duplicating $B(b_R)$ and extending it by adding three more coordinates equal to $(0,1,0)$. Thus, using block vector notation, and with $\allones$ being the all-ones vector, we have
\begin{linenomath}
\begin{align*}
B'(b_0) &= (B(b_R) \mid 0,1,0), \\
B'(b_1) &= (\allones - B(b_L) - B(b_R) \mid 1,0,1), \\
B'(b_2) &= (B(b_L) \mid 0,1,0). 
\end{align*}
\end{linenomath}
The new columns are defined analogously using vectors $B(w_R)$, $\allones-B(w_R)-B(w_L)$, and $B(w_L)$, and extending them as shown in Figure~\ref{fig:pos3ladder}. The condition that the mates of $b_L$ and $b_R$ are disjoint ensures that $\allones-B(b_R) - B(b_L)$ is a $(0,1)$-vector and that every column of $B'_{21}$ sums to one, and analogously for the rows of $B'_{12}$. Note that because $e$ and $f$ are in the auxiliary graph, we know that $w_{R}$ is in the mate of $b_{L}$, and $w_{L}$ is in the mate of $b_{R}$. Thus $B(b_{L},w_{R})=B(b_{R},w_{L})=1$. Since these mates are disjoint, we also see that $B(b_{L},w_{L})=B(b_{R},w_{R})=0$.

Now we must show that for every pair of black vertices $b$, $b^* \in V(G) \cup \{b_0,b_1,b_2\}$ we have
\[
A'(b) \cdot B'(b^*) = 
\begin{cases}
1, & b \not= b^*;\\
2, & b = b^*.
\end{cases}
\]
We break into cases according to whether $b \in \{b_0, b_1, b_2\}$, $b \in \{b_L ,b_R\}$,  or $b \in V(G) \backslash \{b_L ,b_R\}$.

\noindent
{\bf Case 1}: $b \in \{b_0,b_1,b_2\}$.

It is easy to check directly that when both $b, b^* \in \{b_0, b_1, b_2\}$ the dot products have the correct values, so we may take $b^* \in V(G)$.  First assume that $b=b_1$, in which case the dot product $A'(b) \cdot B'(b^*) = B'(b^*, w_0) + B'(b^*,w_1) + B'(b^*,w_2)$ which is equal to one, as it is simply the row-sum of $B'_{12}(b^*)$. Next assume $b = b_0$, in which case the dot product $A'(b) \cdot B'(b^*) = B'(b^*, w_L) + B'(b^*,w_0) + B'(b^*,w_1)$. If $B'(b^{*},w_{L})=1$, then $B'(b^{*},w_{0})=0$, because the mates of $b_{L}$ and $b_{R}$ are disjoint. Similarly $B'(b^{*},w_{1})=0$, so $A'(b) \cdot B'(b^*)=1$. If $B'(b^{*},w_{L})=0$, then exactly one of $B'(b^{*},w_{0})$ and $B'(b^{*},w_{1})$ is equal to one, so again the dot product takes the value one. By symmetry the same holds when $b = b_2$. 

\noindent
{\bf Case 2}: $b \in \{b_L, b_R\}$.

Without loss of generality we assume that $b = b_L$. If $b^{*}\in V(G)$, then the dot product $A'(b) \cdot B'(b^*)$ will be equal to the dot product $A(b) \cdot B(b^*)$, except in the cases where the $(b^*,w_R)$-entry of $B'$ is one (when the dot product will be reduced by one), and where the $(b^*,w_0)$-entry of $B'$ is one (where it will be increased by one). By construction, either both or neither of these occur and so the net result is that $A'(b) \cdot B'(b^*) = A(b) \cdot B(b^*)$. If $b^{*}=b_{0}$, then $A'(b)$ and $B'(b^{*})$ do not share a non-zero entry in the last three columns, or the $w_{L}$ or $w_{R}$ columns. It now follows that $A'(b)\cdot B'(b^{*})=A(b_{L})\cdot B(b_{R})=1$. Assume that $b^{*}=b_{1}$. Note that there is a single column from the last three columns in which $A'(b)$ and $B'(b^{*})$ are both non-zero. In all columns other than the last three, $A'(b)$ contains two non-zeroes. One of these non-zeroes is in the same column as a non-zero of $B(b_{R})$, and the other is in the same column as a non-zero of $B(b_{L})$. From these facts we see that $A'(b)\cdot B'(b^{*})=1$. Finally, if $b^{*}=b_{2}$, then the zero in $A'(b_{L},w_{R})$ means that $A'(b_{L})\cdot B'(b^{*})=A(b_{L})\cdot B(b_{L})-1=1$. So now we have completed the analysis when $b=b_{L}$.

\noindent
{\bf Case 3}: $b \in V(G) \backslash \{b_L, b_R\}$.

The last three coordinates of $A'(b)$ are all zero and so any dot products involving $A'(b)$ can be rewritten using only $A(b)$. If $b^{*} \in V(G)$, then $A'(b) \cdot B'(b^*) = A(b) \cdot B(b^*)$, and because $(A,B)$ is a Lehman pair, these dot products have the required values. If $b^* = b_0$, then $A'(b) \cdot B'(b^*) = A(b) \cdot B(b_R)$ which again has the required value because $(A,B)$ is a Lehman pair. A symmetrical argument holds when $b^{*}=b_{2}$. Finally if $b^* = b_1$ then $A'(b) \cdot B'(b^*) = A(b) \cdot (\allones - B(b_R) - B(b_L))$ which is equal to $3 - 1 - 1$ using the fact that $A$ is cubic, and that the dot-products with the two rows of $B$ are each equal to $1$.
\end{proof}

There is one simple situation where the conditions for inserting a $3$-ladder are automatically satisfied, namely when there is a rung connecting a vertex of $e$ to a vertex of $f$. To show this, we start with a lemma outlining how the mates of two vertices intersect.

\begin{lemma}\label{disjointmates}
Suppose that $A$ is the matrix associated with a cubic Lehman graph, and that $B$ is the $(0,1)$-matrix satisfying $AB^{T}=J+I$. Let $b_1 \ne b_2$ be distinct black vertices, let $w$ be a white vertex and suppose that $A(b_1,w) = A(b_2,w) = 1$ (in other words, $w$ is adjacent to $b_1$ and $b_2$). Then for all $w' \ne w$, either $B(b_1,w')=0$ or $B(b_2,w')=0$. 
\end{lemma}

\begin{proof}
If there is some $w'$ such that $B(b_1,w')= B(b_2,w')=1$, then the $(w',w)$-entry of $B^TA$ is at least $2$ contradicting Theorem~\ref{BridgesRyser}.
\end{proof}

\begin{corollary}
Let $G$ be a $(3s-1,3,s)$ Lehman graph and let $e = (w_L, b_R)$ and $f = (w_R, b_L)$ be edges of $G$. If $(w_L, b_L)$ is a rung of $G$, then $G\mathord{\uparrow}\{e,f\}$ is a $(3(s+1)-1,3,s+1)$ Lehman graph.
\end{corollary}

\begin{proof}
By definition, the edges of $G$ incident with (but not equal to) a rung are in the auxiliary graph, and so satisfy the first condition of Proposition~\ref{prop:positiveexpansion}. Now we show that the mates of $b_L$ and $b_R$ are disjoint. As $A(b_L, w_L) = A(b_R,w_L) = 1$ it follows from Lemma~\ref{disjointmates} that for any white vertex $w \not= w_L$, at least one of $\{B(b_L, w), B(b_R,w)\}$ is zero. As $(b_L,w_L)$ is a rung, $B(b_L,w_L)$ is zero. This covers all white vertices and so no white vertex is in the mate of both $b_L$ and $b_R$. A symmetric argument shows that the mates of $w_L$ and $w_R$ are disjoint.
\end{proof}

The situation for negative cubic Lehman graphs is slightly different in that certain matrix entries must be $0$ rather than $1$, but it is very similar in style, and we omit all but the broadest outline. This construction was previously given by Boros, Gurvich and Hougardy \cite{MR1936943} as the inverse operation of their reduction operation as described in Section~\ref{reduction}.

\begin{proposition}\label{prop:negativeeexpansion}
Let $G$ be a $(3s+1,3,s)$-Lehman graph and let $e = (w_L, b_R)$ and $f = (w_R, b_L)$ be edges of $G$. If $w_L$ is not in the mate of $b_L$ and $w_R$ is not in the mate of $b_R$ and the mates of $b_L$ and $b_R$ are disjoint and the mates of $w_L$ and $w_R$ are disjoint, then $G\mathord{\uparrow}\{e,f\}$ is a $(3(s+1)+1,3,s+1)$-Lehman graph.
\end{proposition}

\begin{proof}
As previously, let $A$, $B$ and $A'$, $B'$ denote the matrices associated with $G$ and $G\mathord{\uparrow}\{e,f\}$. Figure~\ref{fig:neg3ladder} shows how $A'$ is related to $A$ (which is the same as in positive case), and how $B'$ is related to $B$ (which is slightly different to the positive case). The arguments showing that the rows of $A'$ and $B'$ have the ``right'' dot product are entirely analogous to those given in Proposition~\ref{prop:positiveexpansion} for the positive case. 
\end{proof}

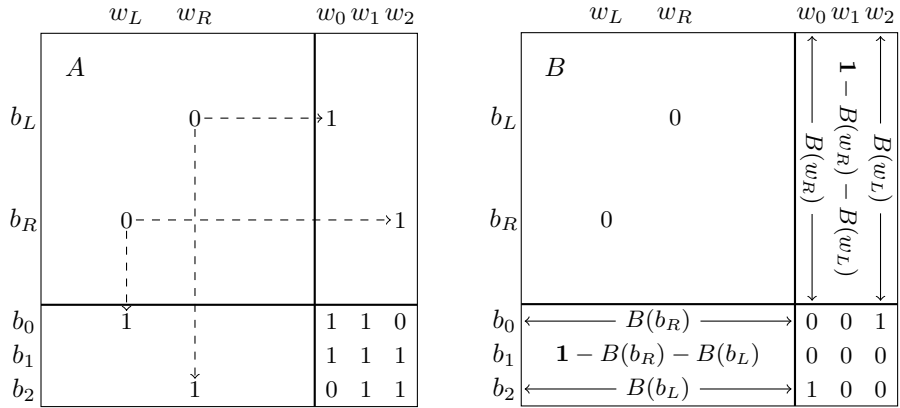
\begin{figure}[htb]
\begin{center}
\begin{tikzpicture}[scale=0.45]

\node at (1,10) {$A$};
\draw (0,0) rectangle (11,11);
\node (wL) at (2.5,11.5){$w_L$};
\node (wR) at (4.5,11.5){$w_R$};
\node (w0) at (8.5,11.5){$w_0$};
\node (w1) at (9.5,11.5){$w_1$};
\node (w2) at (10.5,11.5){$w_2$};

\node (bR) at (-0.5,5.5){$b_R$};
\node (bL) at (-0.5,8.5){$b_L$};
\node (b0) at (-0.5,2.5){$b_0$};
\node (b1) at (-0.5,1.5){$b_1$};
\node (b2) at (-0.5,0.5){$b_2$};

\node [inner sep = 1pt] (wlbr) at (2.5,5.5) {\small $0$};
\node [inner sep = 1pt] (wrbl) at (4.5,8.5) {\small $0$};

\node [inner sep = 1pt] (b0wl) at (2.5,2.5) {\small $1$};
\node [inner sep = 1pt] (b2wr) at (4.5,0.5) {\small $1$};
\node [inner sep = 1pt] (blw0) at (8.5,8.5) {\small $1$};
\node [inner sep = 1pt] (brw2) at (10.5,5.5) {\small $1$};

\draw [dashed, ->] (wlbr) -- (b0wl);
\draw [dashed, ->] (wrbl) -- (b2wr);
\draw [dashed, ->] (wlbr) -- (brw2);
\draw [dashed, ->] (wrbl) -- (blw0);

\draw [thick, black] (8,0) -- (8,11);
\draw [thick, black] (0,3) -- (11,3);

\node at (8.5,0.5) {\small $0$};
\node at (9.5,0.5) {\small $1$};
\node at (10.5,0.5) {\small $1$};
\node at (8.5,1.5) {\small $1$};
\node at (9.5,1.5) {\small $1$};
\node at (10.5,1.5) {\small $1$};
\node at (8.5,2.5) {\small$1$};
\node at (9.5,2.5) {\small $1$};
\node at (10.5,2.5) {\small $0$};

\end{tikzpicture}
\hspace{0.5cm}
\begin{tikzpicture}[scale=0.45]
\node at (1,10) {$B$};
\draw (0,0) rectangle (11,11);
\node (wL) at (2.5,11.5){$w_L$};
\node (wR) at (4.5,11.5){$w_R$};
\node (w0) at (8.5,11.5){$w_0$};
\node (w1) at (9.5,11.5){$w_1$};
\node (w2) at (10.5,11.5){$w_2$};

\node (bR) at (-0.5,5.5){$b_R$};
\node (bL) at (-0.5,8.5){$b_L$};
\node (b0) at (-0.5,2.5){$b_0$};
\node (b1) at (-0.5,1.5){$b_1$};
\node (b2) at (-0.5,0.5){$b_2$};

\draw [thick, black] (8,0) -- (8,11);
\draw [thick, black] (0,3) -- (11,3);

\node at (2.5,5.5) {\small $0$};
\node at (4.5,8.5) {\small $0$};

\node at (8.5,0.5) {\small $1$};
\node at (9.5,0.5) {\small  $0$};
\node at (10.5,0.5) {\small $0$};
\node at (8.5,1.5) {\small $0$};
\node at (9.5,1.5) {\small $0$};
\node at (10.5,1.5) {\small $0$};
\node at (8.5,2.5) {\small $0$};
\node at (9.5,2.5) {\small $0$};
\node at (10.5,2.5) {\small $1$};

\node [rotate = -90] (BwR) at (8.5,7) {\small $B(w_R)$};
\node [rotate = -90] (jw) at (9.5,7) {\small $\allones -B(w_R)-B(w_L)$};
\node [rotate = -90] (BwL) at (10.5,7) {\small $B(w_L)$};

\draw [<-] (8.5,3.1)--(BwR);
\draw [->] (BwR)--(8.5,10.9);

\draw [<-] (10.5,3.1)--(BwL);
\draw [->] (BwL)--(10.5,10.9);

\node (Bbl) at (4,0.5) {\small $B(b_L)$};
\draw [<-] (0.1,0.5)--(Bbl);
\draw [<-] (7.9,0.5)--(Bbl);

\node (jbrbl) at (4,1.5) {\small $\allones -B(b_R)-B(b_L)$};

\node (Bbr) at (4,2.5) {\small $B(b_R)$};
\draw [<-] (0.1,2.5)--(Bbr);
\draw [<-] (7.9,2.5)--(Bbr);

\end{tikzpicture}
\caption{$A'$ and $B'$ after a $3$-ladder insertion in a negative Lehman graph}
\label{fig:neg3ladder}
\end{center}
\end{figure}

\section{Biclique compression and expansion}
\label{completebipartite}

In this section we consider $r$-regular Lehman graphs whose vertex set can be partitioned into copies of $K_{r-1,r-1}$. By a slight abuse of terminology, we will refer to these copies of $K_{r-1,r-1}$ as the \emph{bicliques} of the graph.  (Normally, `biclique' refers to \emph{any} maximal induced complete bipartite subgraph not just those of a particular valency.) We start by showing that any $r$-regular Lehman graph with a partition into bicliques must have $k \in \{-1,1\}$, and moreover must have $r=3$ if $k=-1$. Then each biclique can be compressed to an edge, yielding a smaller Lehman graph that is still $r$-regular, but of the opposite sign. The edges of the smaller graph corresponding to the compressed bicliques form a perfect matching. We then consider the reverse ``expansion'' operation, whereby the edges of a perfect matching are \emph{expanded} into copies of $K_{r-1,r-1}$.  We show that \emph{any} perfect matching in a negative Lehman graph can be expanded to yield a larger Lehman graph with $k=1$. In contrast, if $G$ is a cubic Lehman graph with $k=1$, then the only perfect matching that can be expanded to yield a Lehman graph is the perfect matching of rungs (that is, the set of edges not in the auxiliary graph of $G$).

\subsection{Biclique compression}

First we describe the process of \emph{biclique compression}: Let $G$ be an $r$-regular bipartite graph, and let $\mathcal{P}$ be a partition of $V(G)$ into blocks each of which induces a copy of $K_{r-1,r-1}$.  Each black vertex, $b$, of $G$ is adjacent to the $r-1$ white vertices in its own block, and exactly one additional vertex in a different block; we call this vertex the \emph{out-neighbour} of $b$. Only Lehman graphs with certain parameters may possibly admit a partition into bicliques.

\begin{lemma}
\label{bicliqueconstraints}
Let $G$ be an $r$-regular Lehman graph with $r \geq 3$. If $G$ has a partition into copies of $K_{r-1,r-1}$, then either $k=-1$ and $r=3$, or $k=1$. 
\end{lemma}

\begin{proof}
Suppose first that $k>0$, and consider the mate of a black vertex $b$. As $b$ has a unique out-neighbour, its mate $\mate{}{b}$ contains at least $k$ white vertices from the block containing $b$. But now all the black vertices in this block (of which there are at least two) are dominated at least $k$ times by $\mate{}{b}$ and so $k=1$.

Now suppose that $k=-1$ and that $r > 3$. Let $X$ be a block of $\mathcal{P}$, and let $b_{1}$, $b_{2}$, and $b_{3}$ be distinct black vertices in $X$. Let $w_{1}$ be the out-neighbour of $b_{1}$. Consider the mate, $\mate{}{b_2}$, of $b_{2}$. It cannot contain any white vertex of $X$, since $b_{2}$ is not adjacent with a vertex in its mate. But $b_{1}$ is adjacent with a vertex in $\mate{}{b_2}$, so $w_{1}\in \mate{}{b_2}$. Exactly the same argument shows that $w_{1}$ is in $\mate{}{b_3}$.

Now Proposition \ref{prop4} says that $b_{2}$ and $b_{3}$ are both in the mate of $w_{1}$. Therefore any white vertex in $X$ is adjacent to two vertices in this mate, and this is a contradiction.
\end{proof}

We now define a bipartite graph, $c(G,\mathcal{P})$, that will be the graph obtained by compressing each biclique to an edge. More formally, for each block $X\in \mathcal{P}$, the graph $c(G,\mathcal{P})$ contains a black vertex, $b_{X}$, and a white vertex, $w_{X}$. A black vertex $b_{X}$ is adjacent to a white vertex $w_{Y}$ if and only if $X=Y$, or there is a black vertex in $X$ adjacent to a white vertex in $Y$. (Alternatively, we can see this as the image of $G$ under the graph homomorphism that, for each block $X$, maps the white vertices of $X$ to $w_X$ and the black vertices of $X$ to $b_X$.) The edges $\{\{b_X,w_X\} : X \in \mathcal{P}\}$ form a \emph{perfect matching} in $c(G,\mathcal{P})$. 

\begin{proposition}
\label{posregular}
Let $G$ be an $r$-regular Lehman graph with $k=1$ and $r\geq 3$, and let $\mathcal{P}$ be a partition of $G$ into copies of $K_{r-1,r-1}$. Then $c(G,\mathcal{P})$ is $r$-regular.
\end{proposition}

\begin{proof}
Certainly $c(G,\mathcal{P})$ is a bipartite graph with maximum degree at most $r$, and the number of white vertices is equal to the number of black vertices. Now it is easy to see that if $c(G,\mathcal{P})$ is not $r$-regular, then there is a block $X\in \mathcal{P}$, and distinct black vertices, $b_{1},b_{2}\in X$, such that $b_{1}$ and $b_{2}$ are adjacent with white vertices in the same block, $Y$. Assume that $b_{i}$ is adjacent with $w_{i}\in Y$ for $i=1,2$, and note that $w_{1}\ne w_{2}$.

Let $n$ be the number of black vertices of $G$. Then Theorem \ref{BridgesRyser} says that $rs=n+1$, where $s$ is the number of white vertices in each mate $\mate{}{b}$. Moreover, each white vertex is in exactly $s$ mates of black vertices. If $s=1$, then each black vertex is adjacent with $n+1$ white vertices, which is clearly impossible. Therefore $s>1$. Hence we can choose a set $\Gamma(b)$ such that $w_{1}\in \Gamma(b)$ and $b\ne b_{1}$. Now $\Gamma(b)$ does not contain a vertex in $X$, for otherwise $b_{1}$ would be adjacent with both that vertex, and with $w_{1}$, which is impossible as $b_{1}$ is not $b$. Now $b_{2}$ must be adjacent with at least one vertex in $\Gamma(b)$, so it follows that $w_{2}$ is in $\Gamma(b)$. We choose a black vertex $b'\in Y$ that is not $b$. Therefore $b'$ is adjacent with both $w_{1}$ and $w_{2}$, and these vertices are in $\Gamma(b)$, so we have a contradiction.
\end{proof}

Lemma~\ref{bicliqueconstraints} shows that only cubic graphs can occur in the analogous result for negative Lehman graphs. 

\begin{proposition}
\label{negregular}
Let $G$ be a cubic negative Lehman graph, and let $\mathcal{P}$ be a partition of $G$ into copies of $K_{2,2}$. If $G$ is not the graph produced from $K_{4,4}$ by removing a perfect matching, then $c(G,\mathcal{P})$ is cubic.
\end{proposition}

\begin{proof}
We again assume that there is a block $X\in \mathcal{P}$, and distinct black vertices, $b_{1},b_{2}\in X$, such that the out-neighbours of $b_{1}$ and $b_{2}$ are in the same block, $Y$. Let $w_{i}$ be the out-neighbour of $b_{i}$, for $i=1,2$. Let $n$ be the number of black vertices in $G$. Then $3s=n-1$, where $s=|\Gamma(b)|$ for any black vertex $b$. If $s>1$, then we choose a set $\Gamma(b)$ such that $w_{1}\in \Gamma({b})$ and $b\ne b_{2}$. Then $X\cap \Gamma(b)=\emptyset$, for otherwise $b_{1}$ is adjacent with two vertices in $\mate{}{b}$. Because $b_{2}\ne b$, we see that $b_{2}$ is adjacent with exactly one vertex in $\mate{}{b}$, so $w_{2}$ is in $\mate{}{b}$. Then any black vertex in $Y$ is adjacent with two vertices in $\mate{}{b}$. This contradiction means that $s=1$ and hence $n=4$. Thus $G$ is a cubic bipartite graph with eight vertices. It immediately follows that $G$ is isomorphic to the graph produced from $K_{4,4}$ by removing a perfect matching.
\end{proof}

Our next two lemmas, one for the positive case, and one for the negative case, show that biclique compression preserves the property of being a Lehman graph, but reverses the sign.

\begin{lemma}
\label{postoneglehman}
Let $G$ be an $r$-regular Lehman graph with $k=1$ and $r\geq 3$, and let $\mathcal{P}$ be a partition of $G$ into copies of $K_{r-1,r-1}$. Then $c(G,\mathcal{P})$ is an $r$-regular negative Lehman graph.
\end{lemma}

\begin{proof}
Let $b_{X}$ be an arbitrary black vertex in $c(G,\mathcal{P})$. We will prove the existence of a mate for $b_X$, namely a set of white vertices containing no neighbours of $b_X$ and exactly one neighbour of every other black vertex.

Let $b$ be a black vertex of $G$ in the block $X$, and $w$ be the out-neighbour of $b$ which (by definition) is contained in a block $Y$ distinct from $X$. Let $\Gamma_G(w)$ be the mate of $w$ in $G$; this is well-defined by Corollary~\ref{cor1}. So $\Gamma_G(w)$ is a set of black vertices containing two neighbours of $w$ and exactly one neighbour of every other white vertex in $G$. Note that no block of $\mathcal{P}$ contains more than one vertex of $\Gamma_{G}(w)$, for otherwise that block would contain a white vertex, not equal to $w$, that is dominated by two vertices in $\Gamma_{G}(w)$. In particular, $Y$ does not contain two vertices in $\Gamma_{G}(w)$. Therefore $b$ is in $\Gamma_{G}(w)$, along with exactly one vertex in $Y$.

We now describe a set of vertices $\Gamma$ in $c(G,\mathcal{P})$, and then show that it is a mate for $b_X$. The set $\Gamma$ is defined as follows: 
\[
\Gamma = \{ w_Z : Z \in \mathcal{P}, Z\cap \Gamma_G(w) =\emptyset \}.
\]
In other words, $\Gamma$ contains the white vertices of $c(G,\mathcal{P})$ corresponding to the blocks that contain no vertices of $\Gamma_G(w)$.

Now let $Z$ be a block of $\mathcal{P}$ and suppose that its black vertices are $b_1,b_2,\ldots,b_{r-1}$ with out-neighbours $w_1,w_2,\ldots,w_{r-1}$ that lie in blocks $Z_1,Z_2,\ldots,Z_{r-1}$ respectively (see Figure~\ref{fig:blocks}). The key observation is that unless $\{b_i, w_i\} = \{b,w\}$,
\begin{equation}
\label{keyobs}
Z_i \cap \Gamma_G(w) = \emptyset \text{ if and only if } b_i \in \Gamma_G(w).
\end{equation}
This follows from the fact that each white vertex other than $w$ is dominated by a unique black vertex in $\Gamma_G(w)$. Assume that $\{b_i,w_i\} \not= \{b,w\}$. In this case, if $b_i \in \Gamma_G(w)$ then $w_i$ is dominated by $b_i$ so none of the black vertices in $Z_i$ are in $\Gamma_G(w)$. On the other hand, if $b_i \notin \Gamma_G(w)$ then the vertex dominating $w_i$ must lie in $Z_i$ and so $Z_i$ does contain a vertex of $\Gamma_G(w)$. Note that when $\{b_i,w_i\} = \{b,w\}$, we have $Z = X$ and $Z_i = Y$. In this case, $Z_{i}$ contains an element of $\Gamma_{G}(w)$, since $Y$ contains a single vertex of $\Gamma_{G}(w)$, even though $b=b_{i}$ is also in $\Gamma_{G}(w)$.

\begin{figure}[htb]
\begin{center}
\begin{tikzpicture}
\tikzstyle{rowv}=[circle,fill=black,draw=black ,inner sep = 0.6mm]
\tikzstyle{colv}=[circle,fill=white,draw=black ,inner sep = 0.6mm]
\node [colv] (v0) at (0,-0.75) {};
\node [colv] (v1) at (0,0) {};
\node [colv] (v2) at (0,0.75) {};
\node [rowv] (b0) at (1,-0.75) {};
\node [rowv] (b1) at (1,0) {};
\node [rowv] (b2) at (1,0.75) {};
\node [below] at (b0) {\small $b_3$};
\node [above] at (b1) {\small $b_2$};
\node [above] at (b2) {\small $b_1$};
\draw (v0)--(b0);
\draw (v0)--(b1);
\draw (v0)--(b2);
\draw (v1)--(b0);
\draw (v1)--(b1);
\draw (v1)--(b2);
\draw (v2)--(b0);
\draw (v2)--(b1);
\draw (v2)--(b2);

\node [colv] (a0) at (3,-0.75) {};
\node [colv] (a1) at (3,0) {};
\node [colv] (a2) at (3,0.75) {};
\node [rowv] (a3) at (4,-0.75) {};
\node [rowv] (a4) at (4,0) {};
\node [rowv] (a5) at (4,0.75) {};

\draw (a0)--(a3);
\draw (a0)--(a4);
\draw (a0)--(a5);
\draw (a1)--(a3);
\draw (a1)--(a4);
\draw (a1)--(a5);
\draw (a2)--(a3);
\draw (a2)--(a4);
\draw (a2)--(a5);

\node [colv] (c0) at (2,1.25) {};
\node [colv] (c1) at (2,2) {};
\node [colv] (c2) at (2,2.75) {};
\node [rowv] (c3) at (3,1.25) {};
\node [rowv] (c4) at (3,2) {};
\node [rowv] (c5) at (3,2.75) {};

\draw (c0)--(c3);
\draw (c0)--(c4);
\draw (c0)--(c5);
\draw (c1)--(c3);
\draw (c1)--(c4);
\draw (c1)--(c5);
\draw (c2)--(c3);
\draw (c2)--(c4);
\draw (c2)--(c5);

\node [colv] (d0) at (2,-1.25) {};
\node [colv] (d1) at (2,-2) {};
\node [colv] (d2) at (2,-2.75) {};
\node [rowv] (d3) at (3,-1.25) {};
\node [rowv] (d4) at (3,-2) {};
\node [rowv] (d5) at (3,-2.75) {};

\draw (d0)--(d3);
\draw (d0)--(d4);
\draw (d0)--(d5);
\draw (d1)--(d3);
\draw (d1)--(d4);
\draw (d1)--(d5);
\draw (d2)--(d3);
\draw (d2)--(d4);
\draw (d2)--(d5);

\node  at (3.5,2) {$Z_1$};
\node  at (4.5,0) {$Z_2$};
\node  at (3.5,-2) {$Z_3$};
\node  at (-0.5,0) {$Z$};

\draw (b2) -- (c0);
\draw (b1) -- (a1);
\draw (b0) -- (d0);

\node [above left] at (c0) {\small $w_1$};
\node [above left] at (a1) {\small $w_2$};
\node [above] at (d0) {\small $w_3$};

\end{tikzpicture}
\end{center}
\caption{Configuration in biclique compression}
\label{fig:blocks}
\end{figure}
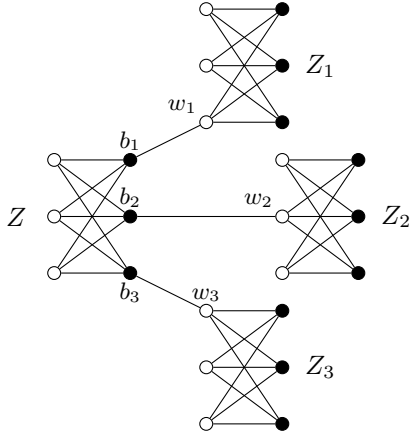

In the compressed graph $c(G,\mathcal{P})$, the neighbours of $b_Z$ are $w_{Z_1},w_{Z_2},\ldots,w_{Z_{r-1}}$ along with $w_Z$, so we must count how many of these vertices are in $\Gamma$. First consider the case where $Z \neq X$. In this case, \eqref{keyobs} implies that if $Z \cap \Gamma_G(w) = \emptyset$, then $w_Z \in \Gamma$, but none of $w_{Z_1},w_{Z_2},\ldots,w_{Z_{r-1}}$ are in $\Gamma$ and so $b_Z$ is adjacent to exactly one vertex in $\Gamma$. On the other hand, if $Z \cap \Gamma_G(w) = \{b_i\}$, then $w_Z$ is not in $\Gamma$. On the other hand, $Z_{i}$ contains no vertex in $\Gamma_{G}(w)$ by \eqref{keyobs}. Therefore $w_{Z_{i}}$ is in $\Gamma$. Because $Z$ contains exactly one vertex in $\Gamma_{G}(w)$, no vertex $b_{j}$ is in $\Gamma_{G}(w)$ when $j\ne i$. Therefore $Z_{j}\cap \Gamma_{G}(w)$ is non-empty by \eqref{keyobs}, so $w_{Z_{j}}$ is not in $\Gamma$. Therefore exactly one neighbour of $b_{Z}$ is in $\Gamma$.

Finally we consider the case that $Z = X$. We assume $b_i = b$, in which case $Z_i = Y$. As both $X$ and $Y$ contain a vertex of $\Gamma_G(w)$, it follows that $w_X,w_Y \notin \Gamma$. Since $Z$ contains a unique vertex of $\Gamma_{G}(w)$, we see that $b_{j}$ is not in $\Gamma_{G}(w)$ when $j\ne i$. Therefore \eqref{keyobs} implies that $Z_{j}$ contains an vertex of $\Gamma_G(w)$. This means that none of the white vertices $w_{Z_{1}},w_{Z_{2}},\ldots, w_{Z_{r-1}}$ are in $\Gamma$. Therefore $\Gamma$ dominates every black vertex of $c(G,\mathcal{P})$ other than $b_X$ exactly once each and does not dominate $b_X$ at all, and so is a mate for $b_X$. As every black vertex of $c(G,\mathcal{P})$ has a mate, and as each mate fails to dominate a unique vertex, it follows that $c(G,\mathcal{P})$ is a negative Lehman graph. 
\end{proof}

\begin{lemma}
\label{negtoposlehman}
Let $G$ be a cubic negative Lehman graph, and let $\mathcal{P}$ be a partition of $G$ into copies of $K_{2,2}$. Then $c(G,\mathcal{P})$ is a cubic Lehman graph with $k=1$.
\end{lemma}

\begin{proof}
This proof is almost identical to that of Lemma~\ref{postoneglehman}. Once again, we select a block $X$, and then construct a set of vertices that will be a mate for $b_X$.   As before, fix a vertex $b \in X$, and consider its out-neighbour $w$ in a block $Y$. Then take the mate $\Gamma_G(w)$. Note that no block contains more than one vertex of $\Gamma_{G}(w)$. Define a set $\Gamma$ of vertices of $c(G,\mathcal{P})$ by
\[
\Gamma = \{ w_Z : Z \in \mathcal{P}, Z\cap \Gamma_G(w) =\emptyset \}.
\]

As $\Gamma_G(w)$ does not contain any neighbours of $w$, it does not contain $b$ nor any of the vertices of $Y$. We need an analogue of \eqref{keyobs}, so suppose that $Z$ is a block containing black vertices $b_1$, $b_2$ with out-neighbours, $w_1$ and $w_2$ respectively, that lie in blocks $Z_1$ and $Z_2$ respectively. Then unless $\{b_i,w_i\} = \{b,w\}$, 
\begin{equation}
Z_i \cap \Gamma_G(w) = \emptyset \text{ if and only if } b_i \in \Gamma_G(w).
\end{equation}
In the exceptional case, when $Z = X$ and $Z_i = Y$, neither $X$ nor $Y$ contain any vertices of $\Gamma_G(w)$.

Now we determine which neighbours of $b_Z$ are in $\Gamma$, considering first the case that $Z \not= X$. If $Z \cap \Gamma_G(w) = \emptyset$, then $w_Z \in \Gamma$, but neither of $w_{Z_1}$, $w_{Z_2} \in \Gamma$, and so $b_Z$ has a unique neighbour. If $Z \cap \Gamma_G(w) = \{b_i\}$ for some $i$, then $w_{Z_i}$ is the unique neighbour of $b_Z$ in $\Gamma$.

Finally, both $X$ and $Y$ contain no vertices of $\Gamma_G(w)$ and hence $b_X$ is adjacent to both $w_X$ and $w_Y$, thereby being twice dominated as required.
\end{proof}

\subsection{Biclique expansion}

Next, we show how to reverse the construction discussed in the previous results in this section, that is, we determine when it is possible to replace each edge in a perfect matching with a biclique while preserving the property of being a Lehman graph.

Let $G$ be an $r$-regular bipartite graph with black vertices $b_{1},\ldots, b_{n}$, and white vertices $w_{1},\ldots, w_{n}$. Assume that $\mathcal{M}=\{b_{i}w_{i}\colon 1\leq i\leq n\}$ is a perfect matching of $G$. Now we define the bipartite graph $e(G,\mathcal{M})$, which will be the \emph{biclique-expansion} of $G$. For each  black vertex $b_i$ of $G$ there is a set $B_i$ of $r-1$ black vertices, and for each white vertex $w_j$, a set $W_j$ of $r-1$ white vertices. For each $i$, join every vertex of $B_i$ to every vertex of $W_i$ so that the subgraphs induced by $B_i \cup W_i$ are all bicliques. Now, for every edge not in $\mathcal{M}$, say $\{b_i,w_j\}$ where $i\not= j$, we add a single edge between $B_i$ and $W_j$ in such a way that $e(G,\mathcal{M})$ is $r$-regular. To see that this can always be done, consider the following procedure: for each black vertex $b$, arbitrarily order the edges of $G$ that are incident with $b$, but not in $\mathcal{M}$. Similarly, arbitrarily order the non-matching edges incident with each white vertex. Now consider an edge $e = \{b,w\}$ which is a non-matching edge of $G$, with corresponding blocks $B$, $W$. If $e$ is the $i$th edge in the ordering for $b$ and the $j$th edge in the ordering for $w$, then join the $i$th vertex of $B$ to the $j$th vertex of $W$. Then $e(G,\mathcal{M})$ is $r$-regular as the $r-1$ edges of $G$ incident with a given vertex $b$ correspond to $r-1$ edges of $e(G,\mathcal{M})$ each using a different vertex of $B$. Moreover, as the vertices of each colour within a biclique can be permuted arbitrarily without altering the isomorphism class of the whole graph, every choice of the edge-ordering at each vertex gives an isomorphic graph. 

This construction is illustrated in Figure \ref{expandingmatching}. The orderings of the edges at each vertex are given in the natural left-to-right order; for example, the two non-matching edges incident with $b_2$ are ordered $(\{b_2,w_1\}, \{b_2,w_4\})$, while the two non-matching edges incident with $w_4$ are ordered $(\{b_2,w_4\}, \{b_3,w_4\})$. So the edge between $B_2$ and $W_4$ connects the second vertex of $B_2$ to the first vertex of $W_4$.

\begin{figure}[htb]

\begin{center}

\begin{tikzpicture}

\tikzstyle{rowvertex}=[circle,fill=black,draw=black ,inner sep = 0.8mm]
\tikzstyle{colvertex}=[circle,fill=white,draw=black ,inner sep = 0.8mm]

\begin{scope}[xshift=0cm]

\node [rowvertex] (b1) at (0,2.5) {};
\node [rowvertex] (b2) at (1,2.5) {};
\node [rowvertex] (b3) at (2,2.5) {};
\node [rowvertex] (b4) at (3,2.5) {};

\node [colvertex] (w1) at (0,0) {};
\node [colvertex] (w2) at (1,0) {};
\node [colvertex] (w3) at (2,0) {};
\node [colvertex] (w4) at (3,0) {};

\draw (b1)--(w1);
\draw (b1)--(w2);
\draw (b1)--(w3);

\draw (b2)--(w1);
\draw (b2)--(w2);
\draw (b2)--(w4);

\draw (b3)--(w1);
\draw (b3)--(w3);
\draw (b3)--(w4);

\draw (b4)--(w2);
\draw (b4)--(w3);
\draw (b4)--(w4);

\node [above] at (b1.north) {$b_1$};
\node [above] at (b2.north) {$b_2$};
\node [above] at (b3.north) {$b_3$};
\node [above] at (b4.north) {$b_4$};

\node [below] at (w1.south) {$w_1$};
\node [below] at (w2.south) {$w_2$};
\node [below] at (w3.south) {$w_3$};
\node [below] at (w4.south) {$w_4$};

\end{scope}

\begin{scope}[xshift=5cm]

\node [rowvertex] (b11) at (0,2.5) {};
\node [rowvertex] (b12) at (0.75,2.5) {};
\node [rowvertex] (b21) at (1.75,2.5) {};
\node [rowvertex] (b22) at (2.5,2.5) {};
\node [rowvertex] (b31) at (3.5,2.5) {};
\node [rowvertex] (b32) at (4.25,2.5) {};
\node [rowvertex] (b41) at (5.25,2.5) {};
\node [rowvertex] (b42) at (6,2.5) {};

\node [colvertex] (w11) at (0,0) {};
\node [colvertex] (w12) at (0.75,0) {};
\node [colvertex] (w21) at (1.75,0) {};
\node [colvertex] (w22) at (2.5,0) {};
\node [colvertex] (w31) at (3.5,0) {};
\node [colvertex] (w32) at (4.25,0) {};
\node [colvertex] (w41) at (5.25,0) {};
\node [colvertex] (w42) at (6,0) {};

\draw[rounded corners] (w11) ++ (-0.25,-0.4) rectangle + (1.25,0.8) node[midway,below=0.5cm]{$W_{1}$};
\draw[rounded corners] (w21) ++ (-0.25,-0.4) rectangle + (1.25,0.8) node[midway,below=0.5cm]{$W_{2}$};
\draw[rounded corners] (w31) ++ (-0.25,-0.4) rectangle + (1.25,0.8) node[midway,below=0.5cm]{$W_{3}$};
\draw[rounded corners] (w41) ++ (-0.25,-0.4) rectangle + (1.25,0.8) node[midway,below=0.5cm]{$W_{4}$};

\draw[rounded corners] (b11) ++ (-0.25,-0.4) rectangle + (1.25,0.8) node[midway,above=0.5cm]{$B_{1}$};
\draw[rounded corners] (b21) ++ (-0.25,-0.4) rectangle + (1.25,0.8) node[midway,above=0.5cm]{$B_{2}$};
\draw[rounded corners] (b31) ++ (-0.25,-0.4) rectangle + (1.25,0.8) node[midway,above=0.5cm]{$B_{3}$};
\draw[rounded corners] (b41) ++ (-0.25,-0.4) rectangle + (1.25,0.8) node[midway,above=0.5cm]{$B_{4}$};

\draw[gray] (b11) -- (w11) -- (b12) -- (w12) -- (b11);
\draw[gray] (b21) -- (w21) -- (b22) -- (w22) -- (b21);
\draw[gray] (b31) -- (w31) -- (b32) -- (w32) -- (b31);
\draw[gray] (b41) -- (w41) -- (b42) -- (w42) -- (b41);

\draw[thick] (b11) -- (w21);
\draw[thick] (b12) -- (w31);

\draw[thick] (b21) -- (w11);
\draw[thick] (b22) -- (w41);

\draw[thick] (b31) -- (w12);
\draw[thick] (b32) -- (w42);

\draw[thick] (b41) -- (w22);
\draw[thick] (b42) -- (w32);

\end{scope}

\end{tikzpicture}

\end{center}

\caption{The $r$-regular bipartite graph, $G$, and $e(G,\mathcal{M})$.}

\label{expandingmatching}

\end{figure}
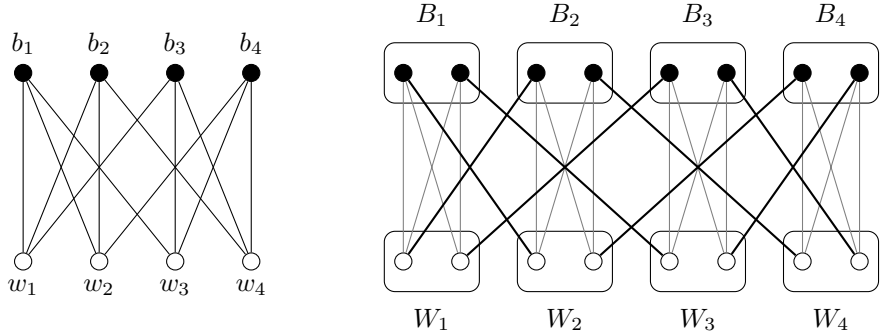

\begin{lemma}
\label{negperfectmatching}
Let $\mathcal{M}$ be a perfect matching of the $r$-regular negative Lehman graph $G$. Then $e(G,\mathcal{M})$ is an $r$-regular Lehman graph with $k=1$.
\end{lemma}

\begin{proof}
Assume that $G$ has black vertices $b_1,b_2,\ldots,b_n$ and white vertices $w_1,w_2,\ldots,w_n$ and that $\mathcal{M} = \{b_{i}w_{i}\colon 1\leq i\leq n\}$. As above, for each $i$, let $B_i$ denote the set of $(r-1)$ black vertices associated with $b_i$ and $W_i$ the set of white vertices  associated with $w_i$.  This is illustrated in Figure~\ref{fig:mate1}, where the complete connection between $B_i$ and $W_i$ is represented by a zigzag line.

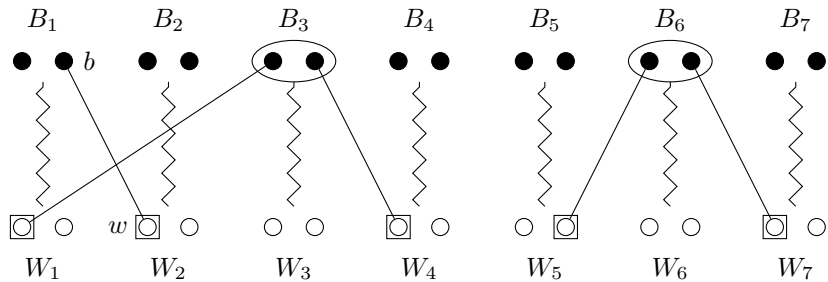
\begin{figure}[htb]
\begin{center}
\begin{tikzpicture}[scale=1.1]
\tikzstyle{rowvertex}=[circle,fill=black,draw=black ,inner sep = 0.8mm]
\tikzstyle{colvertex}=[circle,fill=white,draw=black ,inner sep = 0.8mm]

\foreach \x in {0,1, 3,4, 6, 7, 9, 10, 12, 13, 15, 16, 18, 19} {
\node [rowvertex] (b\x) at (\x/2,2) {};
\node [colvertex] (w\x) at (\x/2,0) {};
}

\foreach \x in {0,3,9,13,18}{
\draw (\x/2-0.14,-0.14) rectangle ++(0.28,0.28);
}

\foreach \y in {0.25, 1.75, 3.25, 4.75, 6.25, 7.75, 9.25} {
\draw decorate [decoration={zigzag}] {(\y,0.25) -- (\y,1.75)};

}

\node at (0.25,2.5) {$B_1$};
\node at (1.75,2.5) {$B_2$};
\node at (3.25,2.5) {$B_3$};
\node at (4.75,2.5) {$B_4$};
\node at (6.25,2.5) {$B_5$};
\node at (7.75,2.5) {$B_6$};
\node at (9.25,2.5) {$B_7$};

\node at (0.25,-0.5) {$W_1$};
\node at (1.75,-0.5) {$W_2$};
\node at (3.25,-0.5) {$W_3$};
\node at (4.75,-0.5) {$W_4$};
\node at (6.25,-0.5) {$W_5$};
\node at (7.75,-0.5) {$W_6$};
\node at (9.25,-0.5) {$W_7$};

\node[right, outer sep = 4pt] at (b1) {$b$};
\node[left, outer sep = 4pt] at (w3) {$w$};

\draw (b1)--(w3);

\draw (7.75,2) ellipse (0.5cm and 0.25cm) {};
\draw (3.25,2) ellipse (0.5cm and 0.25cm) {};

\draw (b15)--(w13);
\draw (b16)--(w18);

\draw (b6)--(w0);
\draw (b7)--(w9);

\end{tikzpicture}
\end{center}
\caption{Finding the mate of $b$ in $e(G,\mathcal{M})$ when $G$ is negative}
\label{fig:mate1}
\end{figure}

Now let $b$ be an arbitrary black vertex of $e(G,\mathcal{M})$. We will show that there is a set $\Gamma$ of white vertices of $e(G,\mathcal{M})$ that is a mate for $b$. Let $w$ be the out-neighbour of $b$ and assume that $b \in B_i$ and $w \in W_j$, for some $i \not= j$. In Figure~\ref{fig:mate1} we have taken $i=1$ and $j=2$ but this is purely for illustrative purposes, and the argument only requires that $i \not= j$.

The vertex $w_j$ in $G$ has a mate $\Gamma_G(w_j)$, which is a set of black vertices of $G$ that dominates every white vertex exactly once, except for $w_j$, which is not dominated at all. Now from this mate, define a set $\Gamma$ of white vertices of $e(G,\mathcal{M})$ by taking the out-neighbours of the vertices in any set $B_\ell$ such that $b_\ell \in \Gamma_G(w_j)$ and adding the vertex $w$. In Figure~\ref{fig:mate1}, the marked sets $B_3$ and $B_6$ correspond to the vertices in $\Gamma_G(w_j)$, and so it is their out-neighbours, together with $w$, that form the purported mate of $b$. Note that no white vertex of $G$ is dominated twice by $\Gamma_{G}(w_{j})$, which means that no set $W_{\ell}$ contains more than one vertex of $\Gamma$.

There are three types of black set, namely the set $B_i$ containing $b$, the sets corresponding to the vertices in $\Gamma_G(w_j)$ (that is, $B_3$ and $B_6$ in Figure~\ref{fig:mate1}) and the remaining sets. Note that $w_{i}\ne w_{j}$, so $w_{i}$ is adjacent to exactly one vertex in $\Gamma_{G}(w_{j})$, say $b_{\ell}$. Furthermore, $\ell\ne i$, for otherwise $w_{j}$ is dominated by a vertex in $\Gamma_{G}(w_{j})$, which is impossible. Therefore $W_{i}$ contains an out-neighbour of a vertex in $B_{\ell}$, so $W_{i}$ contains a vertex in $\Gamma$. All of the black vertices in $B_i$ are dominated by the vertex in $W_i\cap \Gamma$. In addition, $b$ (alone) is dominated a second time by $w$. Let $\ell$ be different from $i$. Assume that $b_\ell$ is in $\Gamma_G(w_j)$. If $W_\ell$ contains a vertex in $\Gamma$, then that vertex is an out-neighbour of a set that corresponds to a member of $\Gamma_{G}(w_{j})$. But this would mean that $w_{\ell}$ is dominated by two vertices in $\Gamma_{G}(w_{j})$, an impossibility. Therefore $W_{\ell}\cap \Gamma=\emptyset$, but each vertex of $B_\ell$ is dominated by its unique out-neighbour. On the other hand, if $b_\ell \notin \Gamma_G(w_j)$, then $\Gamma$ does contain exactly one vertex of $W_\ell$, but does not contain the out-neighbour of any of the vertices in $B_\ell$, so every vertex of $B_\ell$ is dominated exactly once by the sole vertex in $W_i \cap \Gamma$.

Therefore, every black vertex of $e(G,\mathcal{M})$ is dominated once except for $b$, which is dominated twice, and so $\Gamma$ is a mate for $b$. As a mate can be found for every black vertex of $e(G,\mathcal{M})$, it follows that it is a Lehman graph with $k=1$.
\end{proof}

When we expand a Lehman graph satisfying $k=1$, we must restrict to the cubic case, by Proposition \ref{bicliqueconstraints}. In addition, not every perfect matching can be expanded, but only the one formed from the edges not in the auxiliary graph, which we previously termed the \emph{rungs}. 

\begin{lemma}
\label{posperfectmatching}
Let $G$ be a cubic Lehman graph with $k=1$, and let $\mathcal{R}$ be the perfect matching of rungs. Then $e(G,\mathcal{R})$ is a cubic negative Lehman graph.
\end{lemma}

\begin{proof}
We use the notation from Lemma~\ref{negperfectmatching}, and use Figure~\ref{fig:mate2} as an analogous figure to Figure~\ref{fig:mate1}. As before, let $b$ be an arbitrary black vertex lying in a set $B_i$ and let $w$ be its out-neighbour lying in block $W_j$. In Figure~\ref{fig:mate2}, $b \in B_2$ while $w \in W_3$, but the argument only requires that $i \not= j$.

\begin{figure}[htb]
\begin{center}
\begin{tikzpicture}[scale=1]
\tikzstyle{rowvertex}=[circle,fill=black,draw=black ,inner sep = 0.8mm]
\tikzstyle{colvertex}=[circle,fill=white,draw=black ,inner sep = 0.8mm]

\foreach \x in {0,1, 3,4, 6, 7, 9, 10, 12, 13, 15, 16, 18, 19, 21, 22} {
\node [rowvertex] (b\x) at (\x/2,2) {};
\node [colvertex] (w\x) at (\x/2,0) {};
}

\foreach \x in {1,6,7,12,16,21}{
\draw (\x/2-0.14,-0.14) rectangle ++(0.28,0.28);
}
 
\foreach \y in {0.25, 1.75, 3.25, 4.75, 6.25, 7.75, 9.25, 10.75} {
\draw decorate [decoration={zigzag}] {(\y,0.25) -- (\y,1.75)};
}

\node at (0.25,2.5) {$B_1$};
\node at (1.75,2.5) {$B_2$};
\node at (3.25,2.5) {$B_3$};
\node at (4.75,2.5) {$B_4$};
\node at (6.25,2.5) {$B_5$};
\node at (7.75,2.5) {$B_6$};
\node at (9.25,2.5) {$B_7$};
\node at (10.75,2.5) {$B_8$};

\node at (0.25,-0.5) {$W_1$};
\node at (1.75,-0.5) {$W_2$};
\node at (3.25,-0.5) {$W_3$};
\node at (4.75,-0.5) {$W_4$};
\node at (6.25,-0.5) {$W_5$};
\node at (7.75,-0.5) {$W_6$};
\node at (9.25,-0.5) {$W_7$};
\node at (10.75,-0.5) {$W_8$};

\node [inner sep = 0pt] (bl) at (2.35,2.45) {$b$};
\node [left, outer sep = 4pt] at (w6) {$w$};

\draw [->](bl)--(b4);

\draw (1.75,2) ellipse (0.5cm and 0.25cm) {};
\draw (4.75,2) ellipse (0.5cm and 0.25cm) {};
\draw (9.25,2) ellipse (0.5cm and 0.25cm) {};

\draw (b3)--(w1);
\draw (b4)--(w6);
\draw (b9)--(w7);
\draw (b10)--(w12);
\draw (b18)--(w16);
\draw (b19)--(w21);

\end{tikzpicture}
\end{center}
\caption{Finding the mate of $b$ in $e(G,\mathcal{M})$ when $G$ is positive}
\label{fig:mate2}
\end{figure}
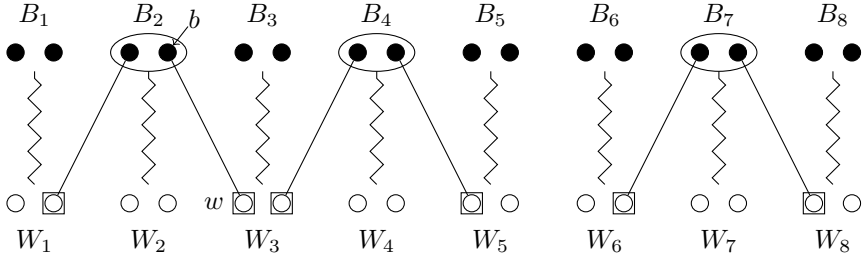

As $G$ is a Lehman graph with $k=1$, the vertex $w_j$ has a mate $\Gamma_G(w_j)$ which is a set of black vertices dominating $w_j$ twice and every other black vertex once. As $\{b_j, w_j\}$ is a rung of $G$, it follows that $b_j \notin \Gamma_G(w_j)$ and hence $b_i \in \Gamma_G(w_j)$. From this mate, we define a set $\Gamma$ of white vertices of $e(G,\mathcal{M})$ by taking the out-neighbours of the vertices in the sets $B_\ell$ provided $b_\ell \in \Gamma_G(w_j)$ and then \emph{removing} the vertex $w$. 

Then arguments essentially identical to those in Lemma~\ref{negperfectmatching} apply unchanged, except for the special status of $b$ and $w$. In summary, for $\ell \neq i$, the vertices of $B_\ell$ are either all dominated by a single vertex in $W_\ell$, or each vertex in $B_{\ell}$ is dominated by its out-neighbour. Which case occurs depends on whether $b_\ell \in \Gamma_G(w_j)$ or not. As $b_i \in \Gamma_G(w_j)$ this general rule indicates that $b \in B_i$ should be dominated by its out-neighbour, which is $w$, but as this has been explicitly excluded from $\Gamma$, the vertex $b$ is the unique undominated black vertex, as required for the mate of $b$ in a negative Lehman graph.   
\end{proof}

We note, without proof, that if $\mathcal{N}$ is any perfect matching of $G$ \emph{other than} the rungs, then although $e(G,\mathcal{N})$ is a well-defined $r$-regular graph, it is not a Lehman graph.

Let $G$ be the graph obtained from $K_{r+1,r+1}$ by deleting a perfect matching. Then the bipartite adjacency matrix $A$ of this graph is equal to $J - I_{r+1}$. If $B = I_{r+1}$, then $AB^T = J-I_{r+1}$ and so $G$ is a negative Lehman graph of type $(r+1,r,1)$. Up to isomorphism, $G$ has a unique perfect matching $\mathcal{M}$ and the expanded graph $e(G,\mathcal{M})$ is a Lehman graph of type $(r^2-1,r,r)$. The clutter matrices of these graphs are those found by Wang \cite{MR2739489}, who also showed that these matrices are minimally nonideal.  

\section{A catalogue of cubic Lehman matrices}
\label{lehmancat}

In this section we describe the computation of a catalogue of cubic Lehman matrices with $k=\pm 1$ thereby partially extending, verifying, and in one instance correcting, the catalogue of L\"{u}tolf \& Margot. Here we recall that our Lehman matrices are square but may have $k=-1$, while their Lehman matrices need not be square, but must have $k > 0$. In addition, they were primarily focussed on (not-necessarily square) mni matrices and so the two catalogues overlap, but are not directly comparable.

We used Gunnar Brinkmann's cubic graph generator \verb+minibaum+ \cite{MR1408342} to generate cubic bipartite graphs on up to $40$ vertices from which to extract the Lehman graphs. Although this is a huge computation, it is possible to prune the generation tree to some extent. For example, if two vertices of degree three in a partially-constructed graph have the same neighbourhood, then any cubic graph constructed by adding additional vertices will have a singular bipartite adjacency matrix, so there is no point in further extending that graph. It is certainly possible to do more sophisticated pruning, but there is a complicated trade-off between computer time, programming time, and the chance of introducing subtle bugs. In the end we opted to modify \verb+minibaum+ by the smallest amount required to make the computation feasible, in the end spending about two months on four $12$-core computers.

There are two natural notions of equivalence for Lehman graphs/matrices one of which is a refinement of the other. In the graph context, we may view two Lehman graphs as being equivalent if and only if there is a \emph{colour-preserving} isomorphism between them. However it is equally natural to just take graph isomorphism (ignoring the vertex colours) as the appropriate concept of equivalence, which we denote \emph{colour-blind isomorphism}. In matrix terms, this corresponds to viewing two matrices as equivalent if one can be obtained from the other by using only row- and column-permutations, or whether matrix transposition is also permitted. The relationship between the two notions of equivalence is straightforward. Each Lehman graph (counted up to graph isomorphism) contributes either $1$ or $2$ to the count of Lehman graphs up to colour-preserving graph isomorphism, depending on whether it has a colour-reversing automorphism or not, respectively. In Tables~\ref{poslehnum} and \ref{neglehnum} the columns $\ell(n)$ and $\ell'(n)$ give the numbers of Lehman graphs up to colour-blind and colour-preserving isomorphism respectively.

\begin{table}
\begin{center}
\begin{tabular}{ccrr}
\hline
$(n,r,s)$ & $2n$ & $\ell(n)$ & $\ell'(n)$ \\
\hline
$(5,3,2)$&$10$&$1$&$1$ \\
$(8,3,3)$&$16$&$2$&$2$ \\
$(11,3,4)$&$22$ &$4$& $4$\\
$(14,3,5)$&$28$&$17$&$18$ \\
$(17,3,6)$&$34$&$71$&$98$ \\
$(20,3,7)$&$40$&$491$&$785$ \\
\hline
\end{tabular}
\end{center}
\caption{Cubic Lehman matrices (with $k=1$) of order at most $20$}
\label{poslehnum}
\end{table}

The numbers given in Table~\ref{poslehnum} differ in only one place from the corresponding numbers found by L\"{u}tolf \& Margot \cite{LM98} --- we find $98$ Lehman graphs (up to colour-preserving isomorphism) of type $(17,3,6)$ compared to the $97$ that they found. By downloading the files associated with their paper, we have confirmed that this is a genuine omission and that our catalogues correspond in all other respects. This omitted graph is shown in Figure~\ref{lmmissing}.

\begin{table}
\begin{center}
\begin{tabular}{ccrr}
\hline
$(n,r,s)$ & $2n$ &  $\ell(n)$ & $\ell'(n)$ \\
\hline
$(4,3,1)$&$8$&$1$&$1$ \\
$(7,3,2)$&$14$&$1$&$1$ \\
$(10,3,3)$&$20$ &$2$&$2$ \\
$(13,3,4)$&$26$&$5$&$5$ \\
$(16,3,5)$&$32$ &$19$&$21$ \\
$(19,3,6)$&$38$&$105$&$154$ \\
$(22,3,7)$&$44$&$853$&$1488$ \\
\hline
\end{tabular}
\end{center}
\caption{Negative cubic Lehman matrices of order at most $22$}
\label{neglehnum}
\end{table}

The numbers given in Table~\ref{neglehnum} coincide with those found by Boros, Gurvich and Hougardy \cite{MR1936943}, with the minor caveat that they were only concerned with the total numbers of $(\alpha,\omega)$-graphs, so did not give values for $\ell(n)$.

\subsection{Cubic mni matrices}

Recall that a clutter matrix is minimally nonideal (mni) if it is not ideal, but all of its proper minors are ideal. A (not-necessarily square) mni matrix $A$ consists of a (necessarily square) Lehman submatrix containing all the rows of minimum weight, say $r$, along with zero or more additional rows of strictly greater weight. It is easy to see that the polyhedron $Q(A)$ has a vertex at $\frac{1}{r} \mathbf{1}$ (the point with all coordinates equal to $\frac{1}{r}$) and \cite{Leh89} showed that if $A$ is mni then this is the \emph{unique} fractional vertex of $Q(A)$.

There are two common ways to represent a polyhedron computationally --- the $H$-representation is a list of inequalities defining half-spaces whose intersection is the polyhedron, and the $V$-representation is a list of the vertices.  In our situation, if $A$ is an $m \times n$ clutter matrix, then $Q(A)$ is defined by $m+n$ inequalities. The first $m$ inequalities, one per row of $A$, are all of the form 
\[
a_{i,0}x_0 + a_{i,1}x_1 + \cdots + a_{i,n}x_n \geq 1,
\]
while the remaining $n$ inequalities, one per column of $A$, are simple non-negativity constraints of the form $x_j \geq 0$. Software is readily available (e.g. in \verb+SageMath+) to convert the $H$-representation of a polyhedron to the $V$-representation (and vice versa if desired). For the sizes we are considering (matrices with around $20$ rows and columns), the process takes no more than a second or so, and as the coordinates of the vertices of $Q(A)$ are rational, it is easy to check how many are fractional.

In this fashion, we can determine which of the Lehman matrices we have constructed are mni themselves (i.e. with no additional rows), and these numbers are shown in Table~\ref{tab:mni} (all these are equivalent to their transpose). This data is consistent with the view that only a small cubic Lehman graph can be mni, and that in general more constraints tend to create more fractional vertices. This seems convincing enough that we are willing to make the following conjecture.

\begin{conjecture}
\label{nocubicmni}
There are no $n \times n$ cubic mni matrices for $n > 17$.
\end{conjecture}

\begin{table}
\begin{center}
\begin{tabular}{cc}
\hline
Parameters & Number \\
\hline
$(5,3,2)$ & $1$ \\
$(8,3,3)$ & $2$ \\
$(11,3,4)$ & $4$ \\
$(14,3,5)$ & $9$ \\
$(17,3,6)$ & $4$ \\
$(20,3,7)$ & $0$ \\
\hline
\end{tabular}
\caption{Numbers of cubic mni Lehman matrices}
\label{tab:mni}
\end{center}
\end{table}

As difficult as it seems to understand square mni matrices, even just cubic ones, the situation is far worse when considering non-square mni matrices. Adding a row to a clutter matrix alters the polyhedron by intersecting it with a new halfspace, thereby both cutting off some existing vertices and adding some new ones. If the new halfspace cuts off more fractional vertices than it creates then (in some sense) the matrix is getting ``closer'' to being mni.  L\"utolf \& Margot used a heuristic based on this general idea to add collections of rows to the $(17,3,6)$ Lehman matrices that they had constructed, and succeeded in extending them to an mni matrix about 30\% of the time. We have not attempted to extend this part of their project.

\section{Projective planes}
\label{ppcore}

The general problem of deciding when a Lehman matrix can be extended to an mni matrix by adding rows seems very difficult. In this section, we consider a very special sub-case of this problem, namely when the Lehman matrix in question is the point-line incidence matrix of a non-degenerate projective plane.

To do this, we need some more notation and background results about clutters and mni matrices. A \emph{transversal} of the clutter $\mathcal{C}=(V,E)$ is a subset of $V$ having non-empty intersection with every element of $E$. The \emph{blocker} of $\mathcal{C}$, written $b(\mathcal{C})$, is the clutter having $V$ as its vertex set, and the set of all minimal transversals of $\mathcal{C}$ as hyperedges. Edwards and Fulkerson made the observation that $b(b(\mathcal{C})) = \mathcal{C}$ \cite{EF70}. The blocker involution exchanges deletion and contraction: $b(\mathcal{C}\backslash v) = b(\mathcal{C})/v$ for all $v\in V$. The blocker of an ideal clutter is also ideal \cite{Leh79} (see also \cite[Theorem 1.17]{Cor01}). Therefore the blocker of an mni clutter is also mni. Let $A$ be an mni clutter matrix. We say that the \emph{core} of $A$, written $\core(A)$, is the submatrix consisting of the minimum weight rows of $A$. For an integer $t\geq 2$, let $J_{t}$ be the clutter $(\{0,1,\ldots,t\},\{\{1,\ldots,t\},\{0,1\},\{0,2\},\ldots,\{0,t\}\})$. Then $J_{t}$ is mni \cite[Exercise 4.2]{Cor01}. We call any clutter isomorphic to $J_{t}$ a \emph{degenerate projective plane}. We can now state the fundamental theorem of minimally nonideal matrices, which is due to Lehman (\cite{Leh89}, see also \cite[Theorem 4.3 and Corollary 4.5]{Cor01}).

\begin{theorem}[Lehman]
\label{Lehman}
Let $\mathcal{C}$ be a minimally nonideal clutter. Let $A$ be the clutter matrix corresponding to $\mathcal{C}$, and let $B$ be the matrix corresponding to $b(\mathcal{C})$. If $\mathcal{C}$ is not a degenerate projective plane, then we can permute rows as necessary so that $(\core(A), \core(B))$ is a Lehman pair. 
\end{theorem}

Any non-degenerate projective plane can be considered as a clutter, where the vertex set is the set of points, and the hyperedges are the lines. These clutters are heavily-studied and highly-structured, so we might hope that it would be possible to classify mni matrices whose core is such a clutter. We make the following conjecture:
\begin{conjecture}
\label{conj:pp}
If $A$ is a minimally nonideal matrix whose core is a non-degenerate projective plane, then $A$ is square and equal to the point-line incidence matrix of the Fano plane.
\end{conjecture}

Novick \cite{Novick} showed that the Fano plane is the only non-degenerate projective plane whose point-line incidence matrix is mni, thus proving Conjecture~\ref{conj:pp} under the assumption that $A$ is square. Her proof was not phrased in geometric terms, so we give an alternative proof that rehearses some of the terminology and ideas we will use later. A \emph{triangle} in a projective plane is the union of three lines that do not share a common point. The \emph{corners} of the triangle are the points in two of the lines.

\begin{proposition}[Novick \cite{Novick}]
\label{trDPP}
Let $A$ be the point-line incidence matrix of a non-degenerate projective plane $\mathcal{P}$. Then $A$ is an mni matrix if and only if $\mathcal{P}$ is the Fano plane.
\end{proposition}

\begin{proof}
One direction of the proof is well known. Suppose that $\mathcal{P}$ is a projective plane of order $k>2$, and let $T$ be a triangle of $\mathcal{P}$. Form a minor of $\mathcal{P}$ (now viewed as a clutter) by first deleting all points not in $T$, and then contracting all the points in $T$ except for the corners. As each line of $\mathcal{P}$ contains at least four points, the only lines contained in $T$ are the three lines of the triangle. This means that after deleting the points not in $T$, we obtain a clutter with exactly three hyperedges. After subsequently contracting the non-corner points, we have a minor isomorphic to the degenerate projective plane $J_{2}$. As $\mathcal{P}$ has a proper minor that is mni, it is not mni itself. 
\end{proof}

Our next results show that no counterexample to Conjecture \ref{conj:pp} can have a core that is the point-line incidence matrix of either the Fano plane $\mathrm{PG}(2,2)$ or the ternary plane $\mathrm{PG}(2,3)$. (In \cite{LM98}, L\"{u}tolf \& Margot denote $\mathrm{PG}(2,3)$ as $L_{13}^{4}(1)$. In a footnote to that paper, they assert without proof the result stated in Theorem \ref{ternarycore}.)

\begin{theorem}
\label{fanocore}
If $A$ is a minimally nonideal clutter matrix whose core is the point-line incidence matrix of the Fano plane $\mathrm{PG}(2,2)$, then $A$ is square.
\end{theorem}

\begin{theorem}
\label{ternarycore}
There is no minimally nonideal clutter matrix whose core is the point-line incidence matrix of $\mathrm{PG}(2,3)$.
\end{theorem}

The proofs of these results rely on the following lemmas. We are hopeful that the geometric approach underlying these lemmas may be further developed. A \emph{blocking set} of a projective plane $\mathcal{P}$ is a minimal set of points that \emph{meets} every line of $\mathcal{P}$, but does not completely \emph{contain} any of the lines (see Bruen \cite{MR0303406}).

\begin{lemma}
\label{lem2}
Let $A$ be a minimally nonideal clutter matrix for a clutter $\mathcal{C}$ whose core is a non-degenerate projective plane 
$\mathcal{P}$. Then 
\begin{enumerate}[label=\textup{(\roman*)}]
\item The blocker $b(\mathcal{C})$ also has $\mathcal{P}$ as its core.
\item Every hyperedge of $\mathcal{C}$ is either a line
of $\mathcal{P}$, or contains a blocking set, but not a line, of $\mathcal{P}$.
\end{enumerate}
\end{lemma}

\begin{proof}
Let the order of $\mathcal{P}$ be $k$. From the properties of projective planes, we know that $\core(A)\core(A)^{T}=J+kI$. Furthermore, Theorem~\ref{Lehman} says that $\core(A)\core(b(A))^{T}=J+k'I$, for some positive integer $k'$. Assume that $\core(b(A))\ne \core(A)$, so that $k'\ne k$. The dot product of any row in $\core(A)$ with the corresponding column of $\core(b(A))^{T}$ is $k'+1$, but the row of $\core(A)$ has weight $k+1$. Therefore $k'<k$. Theorem \ref{BridgesRyser} says that $(k+1)s'=n+k'$, where each row or column sum of $\core(b(A))$ is $s'$. The same theorem says that $(k+1)^{2}=n+k$. This implies that $(k+1)(k+1-s')=k-k'$. But the right side of this equation is positive and less than $k$, and if the left side is positive, it is at least $k+1$. This contradiction proves (i).
As $b(\mathcal{C})$ contains the lines of $\mathcal{P}$, and $b(b(\mathcal{C})) = \mathcal{C}$, it follows that every hyperedge of $\mathcal{C}$ is a transversal to $\mathcal{P}$. Therefore each hyperedge of $\mathcal{C}$ is a line of $\mathcal{P}$ or contains a blocking set of $\mathcal{P}$. As $\mathcal{C}$ is a clutter, no hyperedge of $\mathcal{C}$ can properly contain a line of $\mathcal{P}$.
\end{proof}

For the remainder of this section, we fix some notation as follows. Throughout, $A$ is a minimally nonideal clutter matrix for a clutter $\mathcal{C}$ whose core is a projective plane $\mathcal{P}$. Let $C$ denote the vertices of $\mathcal{C}$ (and hence $\mathcal{P}$), although we will usually call them \emph{points} in our geometric arguments. Through an abuse of notation, we also use $\mathcal{C}$ to refer to the set of hyperedges in $\mathcal{C}$. We frequently argue with respect to an arbitrary triangle $T$ which, unless otherwise stated, is the union of three lines $L_x$, $L_y$ and $L_z$ that pairwise meet in three distinct \emph{corner points} $x = L_y \cap L_z$, $y = L_x \cap L_z$ and $z = L_x \cap L_y$. The non-corner points of the three lines are partitioned into the sets $X=L_{x}-\{y,z\}$, $Y=L_{y}-\{x,z\}$, and $Z=L_{z}-\{x,y\}$. (For convenience, this notation is illustrated in Figure~\ref{fig:tri}.)

\begin{figure}[htb]
\begin{center}
\begin{tikzpicture}[scale=1.5,extended line/.style={shorten >=-#1,shorten <=-#1},
 extended line/.default=.45cm]]
\tikzstyle{point}=[circle,fill=white,draw=black ,inner sep = 0.6mm]
\tikzstyle{smallpoint}=[circle,fill=white,draw=black ,inner sep = 0.4mm]
\coordinate (x) at (90:1cm) {};
\coordinate (y) at (210:1cm) {};
\coordinate (z) at (330:1cm) {};
\draw [extended line] (x)--(y);
\draw [extended line] (y)--(z);
\draw [extended line] (x)--(z);
\draw (x)--(y) node [smallpoint, pos=0.25]{} node [smallpoint, pos=0.5] {} node [smallpoint, pos=0.75]{};
\draw (y)--(z) node [smallpoint, pos=0.25]{} node [smallpoint, pos=0.5] {} node [smallpoint, pos=0.75]{};
\draw (x)--(z) node [smallpoint, pos=0.25]{} node [smallpoint, pos=0.5] {} node [smallpoint, pos=0.75]{};
\node [point] at (x) {};
\node [point] at (y) {};
\node [point] at (z) {};
\node [left] at (x) {$x$};
\node [above left] at (y) {$y$};
\node [above right] at (z) {$z$};
\node at (30:0.8cm) {$L_y$};
\node at (150:0.8cm) {$L_z$};
\node at (270:0.8cm) {$L_x$};

\draw [decorate,decoration={brace,amplitude=4pt}](220:0.7cm) -- (320:0.7cm) ;
\node at (270:0.25cm) {\small $X$};

\end{tikzpicture}
\end{center}
\caption{A triangle in a projective plane}
\label{fig:tri}
\end{figure}
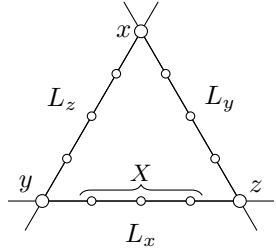

\begin{lemma}
\label{cornertype}
Let $T$ be a triangle of $\mathcal{P}$, and let $R\in \mathcal{C}-\mathcal{P}$ be such that $R\subseteq T$. Then one of the following statements holds.
\begin{enumerate}[label=\textup{(\roman*)}]
\item \label{0corner} $R=X\cup Y\cup Z$.
\item \label{1corner} $R$ has the form $\{x\}\cup Y \cup Z \cup X'$, $\{y\}\cup X \cup Z \cup Y'$, or $\{z\}\cup X \cup Y \cup Z'$, where $X'$, $Y'$, and $Z'$ are, respectively, non-empty subsets of $X$, $Y$, and $Z$.
\item \label{3corner} $R=\{x,y,z\}\cup X' \cup Y'\cup Z'$, where $X'$, $Y'$, and $Z'$ are, respectively, proper subsets of $X$, $Y$, and $Z$.
\end{enumerate}
\end{lemma}

\begin{proof}
This argument falls into cases according to how many of the corners of the triangle are contained in $R$. 

If $R$ contains all three corners of $T$ then because $R$ does not contain a line of $\mathcal{P}$, statement \ref{3corner} holds.

Now suppose that $R$ avoids at least one corner, say $x \notin R$. If $R$ avoids some point $x' \in X$, then $R$ does not intersect the line through $x$ and $x'$, a contradiction to Lemma \ref{lem2} (ii). Thus $X \subseteq R$. Similarly, if $R$ avoids $y$ it contains $Y$ and if it avoids $z$ it contains $Z$. Thus, if $x,y,z \notin R$, statement \ref{0corner} holds.

Now $R$ cannot contain exactly two of $x$, $y$, and $z$, because if it contains, say, $\{x,y\}$ and avoids $z$, then $Z\cup\{x,y\}=L_{z}$ is contained in $R$, contradicting Lemma \ref{lem2}. Up to symmetry, the last case we must consider is when $x \in R$ and $y,z\notin R$. Then $Y\cup Z \subseteq R$. Moreover, $R$ must contain a non-empty set of points in $X$, or it avoids $L_{x}$. Thus statement \ref{1corner} holds.
\end{proof}

Let $T$ be a triangle in $\mathcal{P}$, and let $R \subseteq T$. If $R$ has the form indicated in (i), (ii), or (iii) in Lemma \ref{cornertype}, then we refer to $R$ as, respectively, a \emph{$0$-corner}, \emph{$1$-corner}, or \emph{$3$-corner of $T$}. If $R$ is a $1$-corner containing $x$, then we call $R$ an \emph{$x$-based $1$-corner} (or a $1$-corner \emph{based at $x$}). The terms \emph{$y$-based} and \emph{$z$-based} are defined similarly.

\begin{lemma}
\label{noDPP}
Let $T$ be a triangle of $\mathcal{P}$. Then $\mathcal{C}$ contains the $0$-corner of $T$.
\end{lemma}
 
 \begin{proof}
First suppose that $\mathcal{C}$ does not contain any $0$-corner or $1$-corner of $T$. As in the proof of Lemma~\ref{trDPP}, let $\mathcal{H}$ be the clutter obtained by deleting all points outside $T$, and contracting $X\cup Y\cup Z$. Then the hyperedges of $\mathcal{H}$ are the minimal sets of the form $R-(X\cup Y \cup Z)$, for some $R\in\mathcal{C}$ contained in $T$. By Lemma~\ref{cornertype}, the only hyperedges of $\mathcal{C}$ contained in $T$ are the lines of $T$ and possibly some $3$-corners. The lines of $T$ give rise to the sets $\{x,y\}, \{y,z\}$, and $\{x,z\}$, while every $3$-corner becomes $\{x,y,z\}$ and therefore cannot be minimal. It follows that $\mathcal{H}$ is isomorphic to $J_{2}$. Since $\mathcal{C}$ contains a proper $J_{2}$-minor, it cannot be mni, and we have a contradiction. Therefore $\mathcal{C}$ contains a $0$-corner or a $1$-corner of $T$.

Now suppose that $\mathcal{C}$ does not contain the $0$-corner of $T$, in which case $\mathcal{C}$ at least one $1$-corner of $T$. Without loss of generality, we can assume that $\mathcal{C}$ contains one or more $x$-based $1$-corners of $T$, say $R_1,\ldots, R_d$. Now let $X_i = X\cap R_i$ for $i=1,\ldots, d$, and let $W$ be a minimal subset of $X$ such that $W \cap X_i \neq \emptyset$ for $i=1,\ldots, d$. Since $W$ is minimal, for every $w \in W$ there exists an $i \in \{1,\ldots, d\}$ such that $X_i\cap W=\{w\}$.

Now consider two cases, depending on whether all the $1$-corners of $T$ contained in $\mathcal{C}$ are based at $x$ or not. In both cases we show that $\mathcal{C}$ has a proper minor isomorphic to a degenerate projective plane, contradicting the fact that $A$ is minimally nonideal.

First suppose that all the $1$-corners of $T$ are based at $x$. Define the minor $\mathcal{H}=\mathcal{C} \backslash (C-T)/(Y \cup Z \cup (X-W))$. The sets in $\mathcal{H}$ are the minimal sets of the form $R-(Y \cup Z \cup (X-W))$ for some $R\in\mathcal{C}$ such that $R\subseteq T$. The only members of $\mathcal{C}$ contained in $T$ are the lines of $T$ (which produce the sets $\{x,y\}, \{x,z\}$ and $W$), the $x$-based $1$-corners and possibly some $3$-corners (which all contain $\{x,y\}$ and will therefore not be minimal). The $1$-corners that have exactly one vertex in $W$ produce sets $\{x,w\}$ for every $w\in W$; the other $1$-corners are not minimal. Thus the sets in $\mathcal{H}$ are $\{x,y\}$, $\{x,z\}$, and $\{\{x,w\}\}_{w\in W}$. Thus $\mathcal{H}$ is isomorphic to $J_{|W|+2}$.

Now suppose that $C$ contains a $y$-based $1$-corner, as well as an $x$-based $1$-corner. Define $\mathcal{H}=\mathcal{C} \backslash ((C-T)\cup \{z\}) /(Y \cup Z \cup (X-W))$. The sets in $\mathcal{H}$ arise from members of $\mathcal{C}$ contained in $T-z$. These are $L_{z}$ and the $1$-corners based at $x$ and $y$. The line $L_{z}$ becomes $\{x,y\}$, all the $y$-based $1$-corners become the set $\{y\}\cup W$ and, as before, the $x$-based $1$-corners produce $\{\{x,w\}\}_{w\in W}$. Therefore $\mathcal{H}$ is isomorphic to $J_{|W|+1}$, and the proof is complete.
\end{proof}

\begin{lemma}
\label{doubletriangle}
Assume that the order of $\mathcal{P}$ is greater than two. Then $\mathcal{C}$ is not the clutter obtained from $\mathcal{P}$ by adding all the $0$-corners of triangles of $\mathcal{P}$.
\end{lemma}

\begin{proof}
Assume the lemma fails, so that $\mathcal{C}$ contains the lines of $\mathcal{P}$, the $0$-corners of triangles, and no other hyperedges. Consider three copunctual lines $L_1$, $L_2$, and $L_3$, through the point $x$, and a fourth line, $L_4$, such that $x\notin L_{4}$. Let $X$ be union of these four lines. For $i=1,2,3$, let $y_i$ be the point of intersection of $L_i$ with $L_4$. This configuration contains three triangles $T_1=L_2\cup L_3\cup L_4$, $T_2=L_1\cup L_3\cup L_4$, and $T_3=L_1\cup L_2\cup L_4$. Let $R_1$, $R_2$, and $R_3$ be the $0$-corners of $T_1$, $T_2$, and $T_3$ respectively. For $i=1,2,3$ pick a point $a_i\in L_i -\{x,y_i\}$. The clutter obtained from $\mathcal{C}$ by deleting all points not in $X$ and contracting all points in $X$ other than $\{x,y_1,y_2,y_3,a_1,a_2,a_3\}$ is the Fano plane. Therefore $\mathcal{C}$ has a proper minor isomorphic to an mni clutter, implying that $\mathcal{C}$ is not mni, which contradicts our assumption.
\end{proof}

Now we can prove Theorems \ref{fanocore} and \ref{ternarycore}.

\begin{proof}[Proof of Theorem \ref{fanocore}]
By Lemma \ref{lem2}, any set in $\mathcal{C}-\mathcal{P}$ contains a blocking set of the Fano plane. As the Fano plane has no blocking sets (see Bruen \cite{MR0303406}), it follows that $\mathcal{C}-\mathcal{P}$ is empty and $A$ is the point-line incidence matrix of the Fano plane.
\end{proof}

\begin{proof}[Proof of Theorem \ref{ternarycore}.]
In this case, $\mathcal{C}$ contains at least the lines and the $0$-corners of the ternary plane $\mathrm{PG}(2,3)$. The only blocking sets of $\mathrm{PG}(2,3)$ are the $0$-corners of the triangles (Di Paola \cite{MR0247879}), and as $\mathcal{C}$ is a clutter, it now follows that it contains exactly the lines and the $0$-corner of every triangle and no other sets, contradicting Lemma~\ref{doubletriangle}.
\end{proof}

\section{Open Problems}

In the cubic Lehman graphs described in this paper, $4$-cycles play a major role, either as part of a ladder segment or as part of a biclique partition. While there is nothing in the definition of Lehman graph that immediately implies the existence of $4$-cycles, it seems difficult to find Lehman graphs without them.

\begin{question}
Are there any cubic Lehman graphs (with $k=1$) of girth at least $6$? 
\end{question}

The restriction of this question to $k=1$ is necessary because the Heawood graph and the Desargues graph, which are Lehman graphs of type $(7,3,3)$ and $(10,3,4)$ respectively, both have girth $6$ and $k=2$.

Given any cubic bipartite graph with $2n$ vertices, we can count the number of vertices of one colour that have a valid mate, knowing that the graph is a Lehman graph if and only if this number is $n$. In a Lehman graph every vertex, black or white, has a mate, but if the graph is not Lehman then there may be a different number of black vertices with mates than white ones with mates. None of the cubic bipartite graphs on $17+17$ vertices with girth $6$ have more than six vertices with mates, far short of the $17$ required for the existence of a Lehman graph of that order. 

\begin{question}
Are there any cubic mni Lehman matrices (with $k=1$) of order greater than $17 \times 17$? 
\end{question}

Table~\ref{tab:mni} shows that as $s$ increases, the number of $(3s-1, 3, s)$ Lehman matrices that are mni first increases, reaching a maximum at $s=5$, and then decreases, actually reaching zero when $s=7$. Given that the average number of fractional points in the polytopes $Q(A)$ increases rapidly as the order of $A$ increases, it would not be surprising if square cubic mni matrices only occurred for small orders. Indeed we conjecture that this is so (Conjecture~\ref{nocubicmni}).

\begin{question}
Are there any mni matrices whose core is the point-line incidence matrix of a non-degenerate projective plane of order greater than two? 
\end{question}

If an mni matrix has a non-degenerate projective plane as a core, then the hyperedges all have geometric interpretations as lines or blocking sets, enabling the use of geometric arguments. It would be interesting if such arguments can be pushed further to eliminate more non-degenerate projective planes as potential cores of mni matrices.

Cornu\'{e}jols, Guenin and Tun\c{c}el \cite{MR2507941} call a Lehman matrix \emph{thin} if $k=1$ and asked the following question:

\begin{question}
Are there infinite families of Lehman matrices other than thin matrices and projective planes?
\end{question}

The only known infinite family of Lehman matrices with $k>1$ is the family of point-line incidence matrices of non-degenerate projective planes, which are extremal structures in many ways.  We can make a heuristic argument that we \emph{expect} Lehman matrices with $k>1$ to be far rarer, perhaps vanishingly rare.
If $(A,B)$ is a Lehman pair of type $(n,r,s)$ with $k=r s - n$, then  $\det (A) \det (B) = k^{n-1} (n + k ) = k^{n-1} r s$. If $k=1$, then 
$|\det (A)| = r$ and $|\det (B)| = s$. These are the smallest possible non-zero values for the absolute value of the determinant of an $r$-regular (resp. $s$-regular) matrix, and so there are many such $r$-regular matrices each of which is a candidate to be a Lehman matrix. However if $k > 1$, then there is an extra factor of $k^{n-1}$ in $\det(AB)$, which must be allocated between the determinants of $A$ and $B$. With far far fewer potential Lehman matrices, we are not surprised that the known examples are either small, very highly-structured or both.

\begin{question}
Are there more infinite families of mni Lehman matrices?
\end{question}

For any odd $n$, the circulant matrix with first row $(1,1,0,\cdots,0)$, and its blocker are mni Lehman matrices.
Apart from these, Wang's \cite{MR2739489} ingenious construction provides the only known infinite family of mni Lehman matrices where both $r$, $s > 2$. On one hand, the very existence of such a family makes it seem plausible that there are more, but on the other hand, our structural results make it clear that this family really is very special. It would be interesting to find \emph{any} new square mni Lehman matrix.

\begin{question}
Can similar construction results be developed for higher valency Lehman graphs? 
\end{question}

With respect to this question, it would be interesting to know all the $(15,4,4)$-Lehman graphs. Due to the sheer numbers of bipartite quartic graphs on $30$ vertices, any exhaustive computational approach is likely to require significantly stronger techniques for early pruning of the search, or much stronger constraints on the graph structure. With a heuristic local search based on finding edge-exchanges that increase the number of vertices-with-mates, we have found $58$ Lehman graphs of type $(15,4,4)$ to date. Although we have no good reason to believe that we have covered even a minuscule fraction of the search space, the fact that our searches repeatedly find the same $58$ matrices starting from numerous randomly-chosen bipartite quartic graphs supports the view that this list may be few others. None of the $58$ matrices are mni.

\subsection*{Acknowledgements}

This research was supported by a Rutherford Discovery Fellowship and Australian Research Council Discovery Project DP140102747.

We are indebted to an anonymous referee who informed us that what we knew only as negative Lehman matrices are none other than the partitionable graphs or $(\alpha,\omega$)-graphs that are the central object of study in the search for minimal imperfect graphs.

We also thank Gunnar Brinkmann for his generous advice on how to safely modify {\tt minibaum}, and Clement Lam who kindly unearthed his own copy of the paper \cite{MR593715} when it proved difficult to find through normal library channels.

\providecommand{\bysame}{\leavevmode\hbox to3em{\hrulefill}\thinspace}
\providecommand{\MR}{\relax\ifhmode\unskip\space\fi MR }
\providecommand{\MRhref}[2]{%
  \href{http://www.ams.org/mathscinet-getitem?mr=#1}{#2}
}
\providecommand{\href}[2]{#2}

\end{document}